\documentclass[10pt,reqno]{amsart}
\usepackage{amsmath,amssymb,latexsym,soul,cite,mathrsfs}
\usepackage{color,enumitem,graphicx}
\usepackage[colorlinks=true,urlcolor=blue,
citecolor=red,linkcolor=blue,linktocpage,pdfpagelabels,
bookmarksnumbered,bookmarksopen]{hyperref}
\usepackage[english]{babel}
\usepackage{enumitem}

\usepackage[left=2.2cm,right=2.2cm,top=2.6cm,bottom=2.4cm]{geometry}

\usepackage[hyperpageref]{backref}

\makeatletter
\newcommand{\leqnomode}{\tagsleft@true}
\newcommand{\reqnomode}{\tagsleft@false}
\makeatother

\numberwithin{equation}{section}

\newtheorem{theorem}{Theorem}[section]

\newtheorem{Ex}[theorem]{Example} 
\newtheorem{corollary}[theorem]{Corollary}
\newtheorem{proposition}[theorem]{Proposition}
\newtheorem{cor}[theorem]{Corollary}
\newtheorem{remark}[theorem]{Remark}

\renewcommand{\rightarrow}{\to}
\newcommand{\dive}{\mathrm{div}}

\title[Quasilinear elliptic problems via nonlinear Rayleigh quotient]{Quasilinear elliptic problems via nonlinear Rayleigh quotient}

\author[Edcarlos D. Silva]{Edcarlos D. Silva}
\author[Marcos L. M. Carvalho]{Marcos L. M. Carvalho*}
\author[Leszek Gasi\'nski]{Leszek Gasi\'nski}
\author[Jo\~ao R. Santos J\'unior]{Jo\~ao R. Santos J\'unior}

\address[M. L. M. Carvalho]{Department of Mathematics,
	Federal University of Goi\'{a}s
	\newline\indent
	74001-970, Goi\'{a}s-GO, Brazil}
\email{\href{mailto: marcos$\_$leandro$\_$carvalho@ufg.br}{marcos$\_$leandro$\_$carvalho@ufg.br}}

\address[Leszek Gasi\'nski]{Department of Mathematics, University of the National Education Commission, 
	\newline\indent
	Krakow, Poland}
\email{\href{mailto:leszek.gasinski@up.krakow.pl}{leszek.gasinski@up.krakow.pl}}

\address[Jo\~ao R. Santos J\'unior]{Faculty of Mathematics, Federal university of Par\'a, 
	\newline\indent
	Bel\'em-PA, Brazil}
\email{\href{mailto:joaojunior@ufpa.br}{joaojunior@ufpa.br}}

\address[Edcarlos D. Silva]{Department of Mathematics,
	Federal University of Goi\'{a}s
	\newline\indent
	74001-970, Goi\'{a}s-GO, Brazil}
\email{\href{mailto:edcarlos@ufg.br}{edcarlos@ufg.br.com}}

\thanks{Corresponding author*: Marcos L. M. Carvalho}

\thanks{The first author was partially supported by CNPq with grant 316643/2021-1. The third author was partially supported by CNPq with grant 307045/2021-8. The fourth author was partially supported by CNPq with grant 429955/2018-9.}


\begin{document}
	
	%%%%%%%%%%%%%%%%%%%%%%%%%%%%%%%%%%%%%%%%%%%%%%%%%%%%%%%%%%%%%%%%%%%%%%%%%%%%%%%%%%%%%%%%%%%%%%%%%%%%%%%%%%%%%%%%%%%%%%%%%%%%%%%%%%%%%%%%%%%%%%%%%%%%%%%%%%%%%%%%%%%%%%%%%%%%%%%%%%%%%%%
	%                                                                             ABSTRACT
	%%%%%%%%%%%%%%%%%%%%%%%%%%%%%%%%%%%%%%%%%%%%%%%%%%%%%%%%%%%%%%%%%%%%%%%%%%%%%%%%%%%%%%%%%%%%%%%%%%%%%%%%%%%%%%%%%%%%%%%%%%%%%%%%%%%%%%%%%%%%%%%%%%%%%%%%%%%%%%%%%%%%%%%%%%%%%%%%%%%%%%%
	
	\begin{abstract}
		It is established existence and multiplicity of solution for the following class of quasilinear elliptic problems 
		$$
		\left\{
		\begin{array}{lr}
			-\Delta_\Phi u = \lambda a(x) |u|^{q-2}u + |u|^{p-2}u, & x\in\Omega,\\
			u = 0, & x \in \partial \Omega,
		\end{array}
		\right.
		$$
		where $\Omega \subset \mathbb{R}^N, N \geq 2,$ is a smooth bounded domain, $1 < q < \ell \leq m < p < \ell^*$ and $\Phi: \mathbb{R} \to \mathbb{R}$ is suitable $N$-function. The main feature here is to show whether the Nehari method can be applied to find the largest positive number $\lambda^* > 0$ in such way that our main problem admits at least two distinct solutions for each $\lambda \in (0, \lambda^*)$. Furthermore, using some fine estimates and some extra assumptions on $\Phi$, we prove the existence of at least two positive solutions for $\lambda = \lambda^*$ and $\lambda \in (\lambda^*, \overline{\lambda})$ where $\overline{\lambda} > \lambda^*$. 
		
	\end{abstract}
	\maketitle

	%%%%%%%%%%%%%%%%%%%%%%%%%%%%%%%%%%%%%%%%%%%%%%%%%%%%%%%%%%%%%%%%%%%%%%%%%%%%%%%%%%%%%%%%%%%%%%%%%%%%%%%%%%%%%%%%%%%%%%%%%%%%%%%%%%%%%%%%%%%%%%%%%%%%%%%%%%%%%%%%%%%%%%%%%%%%%%%%%
	%                      	                                      INTRODUCTION
	%%%%%%%%%%%%%%%%%%%%%%%%%%%%%%%%%%%%%%%%%%%%%%%%%%%%%%%%%%%%%%%%%%%%%%%%%%%%%%%%%%%%%%%%%%%%%%%%%%%%%%%%%%%%%%%%%%%%%%%%%%%%%%%%%%%%%%%%%%%%%%%%%%%%%%%%%%%%%%%%%%%%%%%%%%%%%%%%%
	
	\section{Introduction}	
	The present work is devoted to establish existence and multiplicity of solutions to the following class of quasilinear elliptic problem: 
	\begin{equation}\label{pi}
		\left\{
		\begin{array}{lr}
			-\Delta_\Phi u = \lambda a(x) |u|^{q-2}u + |u|^{p-2}u, & x\in\Omega,\\
			u = 0, & x \in \partial \Omega,
		\end{array}
		\right. \tag{$P_\lambda$}
	\end{equation}
	where $\Omega \subset \mathbb{R}^N, N \geq 2$ is a smooth bounded domain, $1 < q < \ell \leq m < p < \ell^*, \ell^* = N \ell/(N - \ell), N \geq 2$. Later on, we shall give the assumptions on $p, q, \Phi$ and the potential $a : \Omega \to \mathbb{R}$. 
	Recall that the $\Phi$-Laplacian operator is given by $\Delta_{\Phi} u:=\mbox{div}(\phi(|\nabla u|)\nabla u)$. Throughout this work we shall consider $\Phi:\mathbb{R}\rightarrow\mathbb{R}$ as an even function defined by
	\begin{equation}\label{phi}
		\Phi(t)=\int_0^{|t|}s\phi(s)\,\mathrm{d} s,
	\end{equation}
	where $\phi:(0,\infty)\to (0,\infty)$ satisfies some suitable conditions.

	Due to the nature of the $\Phi$-Laplacian operator, we need to consider framework of the Orlicz and Orlicz-Sobolev spaces, see Section 2 ahead. These kinds of spaces are appropriated in order to control the generality of the function $\Phi$ which allows us to find weak solutions for the Problem \eqref{pi}. It is important to emphasize that Orlicz and Orlicz-Sobolev spaces was introduced by M. W. Orlicz \cite{Orlicz}. On this subject we also refer the reader to the important works Krasnosel’skii and Rutickii \cite{Krasn}, Luxemburg \cite{Lux} and Donaldson \& Trudinger \cite{Donaldson} and references therein. In those works the main objective is to extend and complement the theory of the Sobolev spaces $W_0^{1,r}(\Omega),~1<r<\infty$. Namely, assuming that $\Phi(t)=|t|^r$, we deduce that the Orlicz-Sobolev space is $W_0^{1,r}(\Omega)$. Furthermore, for the function $\Phi(t)=|t|^r$ we obtain the $p$-Laplacian operator, see \cite{peral}. It is worthwhile to mention also that for $\Phi(t)=|t|^{r_1}+|t|^{r_2}, r_1 < r_2$ where $\Omega$ is a bounded and smooth domain that $W_0^{1,\Phi}(\Omega)=W_0^{1,r_2}(\Omega)$. Moreover, for $\Phi(t)=|t|^{r_1}+|t|^{r_2}$ we obtain the $(r_1,r_2)$-Laplacian operator which is not homogeneous,  see \cite{Olimpio, Li, Marano}. In the same way, assuming that $\Phi(t)=|t|^r\log(1+|t|)$, the associated Orlicz-Sobolev space does not correspond to any Sobolev space, see also \cite{Alberico}. Here we also refer the reader to the important works \cite{adams, Gz1, gossez-Czech, rao} where many aspects on Orlicz and Orlicz-Sobolev spaces are considered.

	It is worthwhile to mention that many special types of quasilinear elliptic problems involving the $\Phi$-Laplacian operator arise from physical phenomena such as in nonlinear elasticity, plasticity, generalized Newtonian fluids, see \cite{fuk2} and references therein. Recall also those semilinear elliptic problems with general nonlinearities have been also studied for a huge class of operators in the last years, see \cite{ABC,brow0,brow1,diegoed,yavdat0}.

	It is important to stress that the nonlinearity $f(x,t) = \lambda a(x) |t|^{q-2}t + |t|^{p-2}t$ is a concave-convex function and the $\Phi$-Laplacian operator is not homogeneous.  Since we are interested in considering general quasilinear elliptic problems we need to deal with some difficulties. The first one is the loss of the homogeneity on the left side hand for the Problem \eqref{pi}. The second difficulty arises whether we consider the largest positive number $\lambda^*$ such that the Nehari method can be applied. To the best of our knowledge, the present work is the first one considering nonhomogeneous operators together with concave-convex nonlinearities where the parameter $\lambda \in (0, \lambda^*)$ plays a principal role. The main difficulty here comes from the fact that for $\lambda=\lambda^*$ the fibering map for the associated energy functional has inflection points. These kinds of difficulties are overcome by using some fine estimates and taking into account the nonlinear Rayleigh quotient method, see \cite{yavdat1, Giacomoni}.  Recall also that the working space is given by the reflexive Banach space $W^{1, \Phi}_0(\Omega)$ where $\Phi$ is a $N$-function. Namely, we consider the Sobolev-Orlicz space which has been widely studied in the last years, see also \cite{Alves1,bonanno,Fuk_1,Fig,fuk2,fang,carvalho,rad1}.

	In the present work our main contribution is to consider the existence and multiplicity of solutions for quasilinear operators driven by the $\Phi$-Laplacian operator where the nonlinearity is concave-convex and the parameter $\lambda > 0$ is not small. Hence, the fibering map for the associated energy functional admits inflections points. Hence, assuming that some minimizer for energy functional restricted to the Nehari manifold is an inflection point, the Lagrange Multiplier Theorem does not work anymore. Hence,  one of the main difficulties is to find a sharp parameter $\lambda^* > 0$ in such way that the Nehari method can be applied for each $\lambda \in (0,\lambda^*)$. Namely, for each $\lambda \in (0, \lambda^*)$, the fibering map does not admit any inflection points. Furthermore, for $\lambda= \lambda^*$, we observe that there exist fibering maps with an inflection point. In order to avoid this kind of problem we use a sequence argument which ensures that our main problem has at least two positive solutions for $\lambda = \lambda^*$. For this case we prove a nonexistence in a subset of the Nehari manifold. More precisely, we prove that the Problem \eqref{pi} does not admit any nontrivial solution which is also an inflection point for the fibering map given just above. The last assertion allows to prove that there exist at least two critical points for our energy functional which are given by minimizers in some subsets of the Nehari manifold. Here we emphasize that in the present work we do not require any kind of homogeneity for the $N$-function $\Phi$. Hence, our work complements the afore-mentioned results by considering the existence and multiplicity of solutions for quasilinear elliptic problems driven by the $\Phi$-Laplacian operator where the homogeneity is not satisfied.

	It is important to stress that the nonlinear Rayleigh quotient method works in general whenever all terms in the energy functional are homogeneous of some degree. In this manner, up to now, the present work is the first one that considers nonhomogeneous operators together with the Nehari method where the nonlinear Rayleigh method can be applied. Furthermore, for $\lambda > \lambda^*$, we prove also existence and multiplicity of solutions using the fact that $\lambda \in (\lambda^*, \overline{\lambda})$ where $\overline{\lambda} > \lambda^*$ is a suitable new parameter.   
	
	\subsection{Assumptions and statement of the main results}
	
	As was told in the introduction we shall consider existence and nonexistence of nontrivial weak solutions for the Problem \eqref{pi} looking for the parameter $\lambda > 0$. From a standard variational point of view finding weak solutions to the elliptic Problem \eqref{pi} is equivalent to finding critical points for the $C^1$ functional $J_{\lambda}: W^{1, \Phi}_0(\Omega) \to \mathbb{R}$ given by
	\begin{equation*}
		J_{\lambda}(u) = \int_{\Omega} \Phi(|\nabla u|) dx  - \frac{\lambda}{q} \int_{\Omega} a(x) |u|^{q} dx -\frac{1}{p}\int_{\Omega} |u|^{p} dx,\, \, u \in W^{1, \Phi}_0(\Omega).
	\end{equation*}	
	It is worthwhile to mention that the Gauteaux derivative of $J_\lambda$ is given by 
	\begin{equation*}
		J_{\lambda}'(u) v = \int_{\Omega} \phi(|\nabla u|) \nabla u \nabla v dx - \lambda \int_{\Omega} a(x) |u|^{q-2}u v dx -\int_{\Omega} |u|^{p-2}u v dx,\, \, u, v \in W^{1, \Phi}_0(\Omega).
	\end{equation*}

	Throughout this work we shall assume the following hypotheses: 
	$\phi:(0,+\infty)\rightarrow(0,+\infty)$ is a $C^1$-function satisfying the following assumptions:
	\begin{itemize}
		\item[($\phi_1$)] $t\phi(t)\mapsto 0$, as $t\mapsto 0$ and $t\phi(t)\mapsto\infty$, as $t\mapsto\infty$;
		\item[($\phi_2$)] $t\phi(t)$ is strictly increasing in $(0, \infty)$;
		\item [($\phi_3$)] $\displaystyle\int_{0}^{1}\frac{\Phi^{-1}(\tau)}{\tau^{\frac{N+1}{N}}}\mathrm{d}\tau<\infty\quad \text{and}\quad
		\displaystyle\int_{1}^{+\infty}\frac{\Phi^{-1}(\tau)}{\tau^{\frac{N+1}{N}}}\mathrm{d}\tau=\infty;$
		\item[$(\phi_4)$] There holds 
		$$-1<\ell - 2 := \inf_{t > 0} \frac{\phi'(t)t}{\phi(t)} \leq \sup_{t > 0} \frac{\phi'(t)t}{\phi(t)} := m - 2<\infty.$$
	\end{itemize}
	Furthermore, we assume also the following assumptions: 
	\begin{enumerate}[label=($H_1$),ref=$(H_1)$] 
		\item \label{paper4f1}
		$1 < q < \ell \leq m < p < \ell^*, \ell^* = N \ell/(N - \ell)$;
	\end{enumerate}

	\begin{enumerate}[label=($H_2$),ref=$(H_2)$] 
		\item \label{paper4f3}
		The function $a: \Omega \to \mathbb{R}$ belongs to $L^{\ell/(\ell-q)}(\Omega)\setminus\{0\}$ and $a(x) \geq a_0 > 0$ for each $x \in \Omega$; 
	\end{enumerate}
	\begin{enumerate}[label=$(H_3)$,ref=$(H_3)$] 
		\item The function \label{paper4f4}
		$$t \mapsto \dfrac{(2 - q) \phi(t) + \phi'(t)t}{t^{p-2}}$$	
		is strictly decreasing for each $t > 0$.	
	\end{enumerate}
	Now, we shall consider the Nehari set as follows
	\begin{equation*}
		\begin{array}{rcl}
			\mathcal{N}_{\lambda}&=&\left\{ u \in  W^{1, \Phi}_0(\Omega)\setminus \{0\}: J_{\lambda}'(u) u  = 0\right\}, \lambda > 0.
		\end{array}
	\end{equation*}
	For the Nehari method we refer the reader to \cite{nehari1,nehari2}. Notice also that the Nehari set can be splited in the following form:
	\begin{eqnarray}
		\mathcal{N}_{\lambda}^+&=&\{u\in\mathcal{N}_\lambda: J''_\lambda(u)(u,u)>0 \}\label{n+}, \nonumber \\ 
		\mathcal{N}_{\lambda}^-&=&\{u\in\mathcal{N}_\lambda: J''_\lambda(u)(u,u)<0 \}\label{n-}, \nonumber \\
		\mathcal{N}_{\lambda}^0&=&\{u\in\mathcal{N}_\lambda: J''_\lambda(u)(u,u)=0\}\label{n_0}. \nonumber 
	\end{eqnarray}	
	It is important to emphasize that $u \mapsto J_{\lambda}''(u)(u,u)$ is well defined for each $u \in W^{1,\Phi}_{0}(\Omega)$. The main feature in the present work is to ensure that the minimization problems
	\begin{equation}\label{ee1}
		c_\mathcal{N^+} :=\inf\{J_{\lambda}(u): u\in \mathcal{N^+}\}
	\end{equation}
	and
	\begin{equation}\label{ee2}
		c_\mathcal{N^-} :=\inf\{J_{\lambda}(u): u\in \mathcal{N^-}\}
	\end{equation}	
	are attained by some specific functions.

	From now on, we shall introduce the nonlinear generalized Rayleigh quotients, see \cite{yavdat1}. Namely, we consider the functionals $R_n,R_e: W^{1,\Phi}_{0}(\Omega) \setminus\{0\} \to \mathbb{R}$ 
	associated with the parameter $\lambda>0$ in the following form 
	\begin{equation}\label{Rn1}
		R_n(u)=\dfrac{\int_{\Omega} \phi(|\nabla u|) |\nabla u|^2 dx- \|u\|_p^p}{\|u\|^q_{q,a}}, \,\, u \in W^{1,\Phi}_{0}(\Omega) \setminus \{0\}
	\end{equation}
	and
	\begin{equation}\label{Re2}
		{R}_e(u)=\dfrac{ \int_{\Omega} \Phi(|\nabla u|) dx - \frac{1}{p}\|u\|_p^p}{\frac{1}{q}\|u\|^q_{q,a}}, \,\, u \in W^{1,\Phi}_{0}(\Omega) \setminus \{0\},
	\end{equation}
	where $\|u\|^q_{q,a}:=\int_\Omega a|u|^qdx$.
	The main idea here is to find $u \in X\setminus \{0\}$ such that $R_n(u) = \lambda$. More specifically, we obtain the following statement:
	\begin{equation}\label{b11}
		u \in \mathcal{N}_{\lambda} \, \, \, \mbox{if and only if} \,\, \,   R_n(u) = \lambda, u \in W^{1,\Phi}_{0}(\Omega) \setminus \{0\}
	\end{equation}
	and 
	\begin{equation}\label{b111}
		J_{\lambda}(u) = 0 \, \, \mbox{if and only if} \,\, \,   R_e(u) = \lambda, u \in W^{1,\Phi}_{0}(\Omega)\setminus \{0\}.
	\end{equation}
	The functional given in \eqref{Re2} is used in order to decide the sign for solution obtained as a minimizer for the functional $J_\lambda$ restricted to $\mathcal{N}_{\lambda}^-$. Under this condition, we need to control the size of $\lambda \in (0, \lambda^*)$ using an auxiliary extreme. To do that we need to consider some extra functionals. Namely, we consider the $C^1$ functionals $\Lambda_e,\Lambda_n:  W^{1,\Phi}_{0}(\Omega) \setminus \{0\} \rightarrow \mathbb{R}$ in the following form:
	\begin{equation*}
		\Lambda_e(u) :=\max_{t>0} R_e(tu) \, \, \mbox{and} \,\,
		\Lambda_n(u) =\max_{t>0} R_n(tu).
	\end{equation*}
	These functionals give us sharp conditions taking into account the Nehari method for the existence of weak solutions $u \in X$ for our main problem with $\lambda \in (0, \lambda^*)$. In order to do that we define the following extreme:	
	\begin{eqnarray}  \label{lambdae}
		\lambda_*:=\inf_{u\in  W^{1,\Phi}_{0}(\Omega)\setminus\{0\}} \Lambda_e(u)
	\end{eqnarray}
	and
	\begin{eqnarray} \label{lambdan}
		\lambda^*:=\inf_{u\in  W^{1,\Phi}_{0}(\Omega) \setminus\{0\}} \Lambda_n(u).
	\end{eqnarray}
	
	\begin{theorem}\label{th1}
		Suppose that assumptions $(\phi_1)-(\phi_4)$ and \ref{paper4f1}--\ref{paper4f4} are satisfied. Then $0 < \lambda_* < \lambda^* < \infty$. Furthermore,  Problem \eqref{pi} admits at least one positive ground state solution $u \in X$ for each $\lambda \in (0, \lambda^*)$. The solutions $u \in X$ satisfies the following properties:
		\begin{itemize}
			\item[i)] $J_{\lambda}''(u)(u, u) > 0, J_{\lambda}(u) < 0, u \in \mathcal{N}_\lambda^+$;
			\item[ii)] There holds $\|u\| \to 0$ as $\lambda \to 0$.
		\end{itemize} 
	\end{theorem}

	\begin{theorem}\label{th2}
		Suppose that assumptions $(\phi_1)-(\phi_4)$ and \ref{paper4f1}--\ref{paper4f4} are satisfied. Then Problem \eqref{pi} admits at least one positive solution $v \in X$ for each $\lambda \in (0, \lambda^*)$. The solution $v \in X$ satisfies also the following properties:
		\begin{itemize}
			\item[i)] $J_{\lambda}''(v)(v, v) < 0, v \in \mathcal{N}_\lambda^-$;
			\item[ii)] There holds $J_{\lambda}(v) > 0$ for each $\lambda \in (0, \lambda_*)$;
			\item[iii)] There holds $J_{\lambda}(v) = 0$ for $\lambda = \lambda_*$;
			\item[iv)] There holds $J_{\lambda}(v) < 0$ for each $\lambda \in(\lambda_*, \lambda^*)$.
		\end{itemize}  
	\end{theorem}
	Now, using Theorems \ref{th1} and \ref{th2} we can state the following result:
	\begin{corollary}\label{cor}
		Suppose that assumptions $(\phi_1)-(\phi_4)$ and \ref{paper4f1}--\ref{paper4f4} are satisfied. Then  Problem \eqref{pi} has at least two positive solutions provided that $\lambda \in (0, \lambda^*)$.
	\end{corollary}
	
	In what follows we shall consider the case $\lambda \in [\lambda^*, \infty)$. In order to do that we shall explore a sequence $(\lambda_k)_{k \in \mathbb{N}} \in \mathbb{R}$ such that $\lambda_k < \lambda^*$ for each $k \in \mathbb{N}$ and $\lambda_k \to \lambda^*$ as $k \to \infty$.  Furthermore, we need to consider some extra assumptions. More specifically, we assume the following extra assumptions:
	\begin{enumerate}
		\item[$(\phi_5)$] The function $S:\mathbb{R}\to [0,\infty)$ defined by $S(t)=p\Phi(t)-\phi(t)t^2, t > 0,$ is convex;
	\end{enumerate}
	\begin{enumerate}[label=($H_4$),ref=$(H_4)$] 
		\item \label{paper4f5} $a\in L^\infty_{loc}(\Omega)$ and $\ell\frac{m-q}{\ell-q}<p<\ell^*$.
	\end{enumerate}
	Under these conditions, we can state our next main result as follows:	
	\begin{theorem}\label{th3}
		Suppose that assumptions $(\phi_1)-(\phi_5)$ and \ref{paper4f1}--\ref{paper4f5} are satisfied. Assume also that $\lambda = \lambda^*$ and Problem \eqref{pi} does not admit weak solutions in the set $\mathcal{N}_\lambda^0$. Then Problem \eqref{pi} admits at least two positive solution. 
	\end{theorem}
	
	At this stage, we shall consider the following positive number
	\begin{equation}\label{lambdabarra}
		\overline{\lambda} = \sup\{ \lambda \in [\lambda^*, \infty) : \,\, \mbox{Problem \eqref{pi} does not admits solution} \,\, u \in \mathcal{N}_{\lambda}^0 \}.
	\end{equation}
	
	\begin{cor}\label{cor2}
		Suppose that assumptions $(\phi_1)-(\phi_5)$ and \ref{paper4f1}--\ref{paper4f5} are satisfied. Assume also $\lambda \in [\lambda^*, \overline{\lambda})$ where $\lambda^* < \overline{\lambda}$. Then Problem \eqref{pi} has at least two positive solutions $u \in \mathcal{N}_{\lambda}^+$ and $v \in \mathcal{N}_{\lambda}^-$.	
	\end{cor}

	\begin{cor}\label{cor3}
		Suppose that assumptions $(\phi_1)-(\phi_5)$ and \ref{paper4f1}--\ref{paper4f5} are satisfied. Assume also that $\lambda^* < \overline{\lambda}$. Then Problem $(P_{\lambda})$ has at least one positive solution whenever $\lambda = \overline{\lambda}$.	
	\end{cor}
	
	\begin{remark}
		Assume \ref{paper4f1}--\ref{paper4f4}. Then, assuming that $\ell = m$, we obtain that $\lambda^* < \overline{\lambda}$. Under these conditions,  we mention that hypothesis \ref{paper4f1} implies that \ref{paper4f5} is satisfied. In general, assuming that $\ell < m$, the inequality $\lambda^* < \overline{\lambda}$ remains as an interesting open problem. The main difficulty here is to ensure that Problem \eqref{pi} does not admit a weak solution $u \in \mathcal{N}_\lambda^0$ with $\lambda \in (\lambda^*, \overline{\lambda})$ where $\overline{\lambda}$ is given by \eqref{lambdabarra}.  
	\end{remark}

	\begin{remark}
		It is important to stress that hypothesis $(\phi_5)$ is satisfied for  several $N$-functions $\Phi$. Namely, assuming that
		$$\ell-2\leq \frac{t(t\phi(t))''}{(t\phi(t))'}\leq m-2 $$
		holds we infer that hypothesis $(\phi_5)$ is verified. For this kind of assumptions we refer the reader to \cite{ed2}.
	\end{remark}	
	
	\begin{Ex}
		It is important to observe that our assumption \ref{paper4f4} is satisfied for the $N$-function $\Phi(t) = |t|^r/r, t > 0$ which generates the $r$-Laplacian operator with $1 < r < \infty$. Now, we consider the function $\phi_{r_1,r_2}(t) = |t|^{r_1-2} + |t|^{r_2 -2}, t \geq 0, 1 < q < r_1 < r_2 < p < \infty$. It is well known that the last function provides us the following quasilinear elliptic problem:
		\begin{equation}\label{aii}
			\left\{
			\begin{array}{lr}
				-\Delta_{r_1} u - \Delta_{r_2} u = \lambda a(x) |u|^{q-2}u + |u|^{p-2}u, & x\in\Omega,\\
				u = 0, & x \in \partial \Omega,
			\end{array}
			\right. \tag{$P_{r_1,r_2}$}
		\end{equation}
		For this function we observe that $\ell = r_1$ and $m = r_2$. It is important to stress that the quasilinear elliptic problem \eqref{aii} is named as the $(r_1,r_2)$-Laplacian operator. Here we refer the interested reader to \cite{Olimpio, Li, Marano} and references therein.  Now, for $\phi(t) = \phi_{r_1,r_2}(t)$ we obtain that
		\begin{equation*}
			t \mapsto	\dfrac{(2 - q) \phi_{r_1,r_2}(t) + \phi_{r_1,r_2}'(t)t}{t^{p-2}} =  (r_1 - q)t^{r_1 -p} + (r_2 - q) t^{r_2-p}, t > 0,
		\end{equation*}
		is a strictly decreasing function. Hence, the function $\phi_{r_1,r_2}$ satisfies our assumption \ref{paper4f4}. Now, we observe that 
		$$S(t)=\left(\frac{p}{r_1}-1\right)t^{r_1}+\left(\frac{p}{r_2}-1\right)t^{r_2}, t \geq 0.$$
		Hence, hypothesis $(\phi_5)$ is also verified. As a product, our main assumption is satisfied for $\phi(t) = \phi_{r_1,r_2}(t)$. As was told before this kind of function generates the so-called $(r_1, r_2)$- Laplacian operator. This operator is not homogeneous due to the fact that $r_1 < r_2$. Once again our assumptions can be applied for general $N$-functions $\Phi$ which are not homogeneous in general.
	\end{Ex}

	\begin{Ex} 
		Now, we consider another interesting example for a $N$-function which satisfies our main assumptions. Firstly, we consider $\phi(t) = \ln(1 + |t|), t \in \mathbb{R}$. For this example we obtain that $\ell = 2$ and $m = 3$ where 
		$$\Phi(t) = \int_{0}^t \phi(s)s ds=\frac{t^2}{2}\ln(1+t)-\frac{t^2}{4}+\frac{t}{2}-\frac{1}{2}\ln(1+t), t \geq 0.$$ 
		Furthermore, hypotheses $(\phi_1) -(\phi_4)$ and \ref{paper4f1} - \ref{paper4f3} are satisfied. Furthermore, we observe that 
		$$t \mapsto \dfrac{(2 - q) \phi(t) + \phi'(t)t}{t^{p-2}} = \dfrac{(2 -q) ln(1 + t)}{t^{p-2}} + \dfrac{1}{t^{p-1}(1 + t)}, t > 0.$$
		It is not hard to see that the function
		$$t \mapsto \dfrac{(2 -q) ln(1 + t)}{t^{p-2}}, t > 0,$$
		is decreasing. Here we recall that $1<q<\ell=2<m=3<p<\ell^*$. Therefore, the hypothesis \ref{paper4f4} is also verified. Furthermore, we observe that 
		$$S(t)=\left(\frac{p}{2}-1\right)t^2\ln(1+t)-\frac{p}{4}t^2+\frac{p}{2}t-\frac{p}{2}\ln(1+t), t \geq 0.$$
		Hence,
		$$S'(t)=2\left(\frac{p}{2}-1\right)t\ln(1+t)+\left(\frac{p}{2}-1\right)\frac{t^2}{1+t}-\frac{p}{2}t+\frac{p}{2}-\frac{p}{2(1+t)}, t \geq 0,$$
		and
		$$S''(t)=\left(\frac{p}{2}-1\right)\left[2ln(1+t)+\frac{2t}{1+t}+\frac{2t+t^2}{(1+t)^2}\right]-\frac{p}{2} + \frac{p}{2(1+t)^2}, t \geq 0.$$
		Now, using the elementary inequality $\frac{t}{t+1}\leq \ln(1+t),~t\geq 0$, we deduce that 
		$$S''(t)\geq \frac{4(p-3)t+2(2p-5)t^2}{2(1+t)^2}\geq 0,~t\geq 0.$$
		Thus, $S$ is a convex function proving that hypothesis $(\phi_5)$ is also satisfied. Moreover,  the hypothesis \ref{paper4f5} is also verified for each $p>2\frac{3-q}{2-q}$. It is important to emphasize that for this kind of $N$-function the Orlicz-Sobolev space $W^{1,\Phi}_0(\Omega)$ does not coincide with any Sobolev space $W^{1,p}_0(\Omega)$ with $p \in (1, \infty)$. A similar assertion holds for the spaces $L_{\Phi}(\Omega)$ and $L^p(\Omega)$ with $p \in (1, \infty)$. For further results on Orlicz and Orlicz-Sobolev spaces, we refer the reader to \cite{Alberico, ed2}. 
	\end{Ex}

	\subsection{Outline}
	The remainder of this paper is organized as follows: In the forthcoming Section we consider a brief review of Orlicz and Orlicz Sobolev spaces. In Section 3 we consider the Nehari manifold method for our main problem. In Section 4 we prove some useful results for our main problem with $\lambda = \lambda^*$. In Section 5 we some result assuming that $\lambda > \lambda^*$. The Section 6 is devoted to the proof of our main results. 
	
	\subsection{Notation} Throughout the paper we will use the following notation: $C$, $\tilde{C}$, $C_{1}$, $C_{2}$,... denote positive constants (possibly different). The norm in $L^{p}(\mathbb{R}^N)$ and $L^{\infty}(\mathbb{R}^N)$, will be denoted respectively by $\|\cdot\|_{p}$ and $\|\cdot\|_{\infty}$. The norm in $W^{1,\Phi}_0(\Omega)$ is given by the Orlicz-Luxemburg norm denoted by $\| \cdot \|$, see Section 2 ahead.  Furthermore, $|\cdot|$ denotes the Lebesgue measure on $\mathbb{R}^{N}$.
	
	\section{Preliminaries for the Orlicz and Orlicz-Sobolev framework}\label{orlicz}
	
	In this section we shall give some powerful tools for the Orlicz and Orlicz-Sobolev spaces which will be used later. Firstly, we recall some basic concepts on Orlicz and Orlicz-Sobolev spaces. For a more complete discussion on this subject, we refer the readers to \cite{adams,Krasn, Gz1, gossez-Czech, rao}. Let $\Theta:\mathbb{R}\rightarrow[0,+\infty)$ be convex and continuous. It is important to say that $\Theta$ is a $N$-function if $\Theta$ satisfies the following conditions:
	\begin{itemize}
		\item[(i)] $\Theta$ is even;
		\item[(ii)] $\displaystyle\lim_{t\rightarrow 0}\frac{\Theta(t)}{t}=0$;
		\item[(iii)] $\displaystyle\lim_{t\rightarrow \infty}\frac{\Theta(t)}{t}=\infty$;
		\item[(iv)] $\Theta(t)>0$, for all $t>0$.	
	\end{itemize}
	Notice that by using assumptions $(\phi_1)$ and $(\phi_2)$ we conclude that $\Phi$, defined in \eqref{phi}, is a $N$-function. Henceforth, $\Theta$ and $\Psi$ denote $N$-functions. 
	
	Recall also that a $N$-function satisfies the $\Delta_{2}$-condition if there exists $K>0$ such that
	\[
	\Theta(2t)\leq K\Theta(t), \quad \mbox{for all} \hspace{0,2cm} t\geq0.
	\]
	We denote by $\tilde{\Theta}$ the complementary function of $\Theta$, which is given by the Legendre's transformation
	\[
	\tilde{\Theta}(s)=\max_{t\geq0}\{st-\Theta(t)\}, \quad \mbox{for all} \hspace{0,2cm} s\geq0.
	\]
	Let $\Omega\subset\mathbb{R}^{N}$ be an open subset and $\Theta:[0,+\infty)\rightarrow[0,+\infty)$ be fixed. The set
	\[
	\mathcal{L}_{\Theta}(\Omega):=\left\{u:\Omega\rightarrow\mathbb{R} \mbox{ measurable } : \int_{\Omega}\Theta(|u(x)|)dx<+\infty\right\},
	\]
	is the so-called \textit{Orlicz class}. Let us suppose that $\Theta$ is a Young function generated by $\theta$, that is
	\[
	\Theta(t)=\int_{0}^{t}s \theta(s)\,\mathrm{d}s.
	\]
	Let us define $\displaystyle\tilde{\theta}(t):=\sup_{s\theta(s)\leq t}s$, for $t\geq0$. The function $\tilde{\Theta}$ can be rewritten as follows
	\[
	\tilde{\Theta}(t)=\int_{0}^{t} s \tilde{\theta}(s)\,\mathrm{d}s.
	\]
	The function $\tilde{\Theta}$ is called the \textit{complementary function} to $\Theta$. The set
	\[
	L_{\Theta}(\Omega):=\left\{u:\Omega\rightarrow\mathbb{R}:\int_\Omega \Theta\left(\frac{|u(x)|}{\lambda}\right)<\infty, \mbox{ for some } \lambda>0 \right\},
	\]
	is called \textit{Orlicz space}. The usual norm on $L_{\Theta}(\Omega)$ is the \textit{Luxemburg norm}
	\[
	\|u\|_\Theta=\inf\left\{\lambda>0~|~\int_\Omega \Theta\left(\frac{|u(x)|}{\lambda}\right) \leq 1\right\}.
	\]
	We recall that the \textit{Orlicz-Sobolev space} $W^{1,\Theta}(\Omega)$ is defined by
	\[
	W^{1,\Theta}(\Omega):=\left\{u\in L_{\Theta}(\Omega):\exists f_{i}\in L_{\Theta}(\Omega), \int_{\Omega}u\frac{\partial \phi}{\partial x_{i}}=-\int_{\Omega}f_{i}\phi, \ \forall \phi\in C^{\infty}_{0}(\Omega), \ i=1,...,N\right\}.
	\]
	The Orlicz-Sobolev norm of $ W^{1, \Theta}(\Omega)$ is given by
	\[
	\displaystyle \|u\|_{1,\Theta}=\|u\|_\Theta+\sum_{i=1}^N\left\|\frac{\partial u}{\partial x_i}\right\|_\Theta.
	\]
	Since $\Theta$ satisfies the $\Delta_{2}$-condition, we define by $W_0^{1,\Theta}(\Omega)$  the closure of $C_0^{\infty}(\Omega)$ with respect to the Orlicz-Sobolev norm of $W^{1,\Theta}(\Omega)$. By the Poincar\'e inequality (see e.g.  \cite{gossez-Czech}), that is, the inequality
	\[
	\int_\Omega\Theta(u)\leq \int_\Omega\Theta(2d_{\Omega}|\nabla u|),
	\]
	where $d_{\Omega}=\mbox{diam}(\Omega)$, we can conclude that
	\[
	\|u\|_\Theta\leq 2d_{\Omega}\|\nabla u\|_\Theta, \quad \mbox{for all} \hspace{0,2cm} u\in W_0^{1,\Theta}(\Omega).
	\]
	As a consequence, we have that  $\|u\| :=\|\nabla u\|_\Theta$ defines a norm in $W_0^{1,\Theta}(\Omega)$ which is equivalent to $\|\cdot\|_{1,\Phi}$. The spaces $L_{\Theta}(\Omega)$, $W^{1,\Theta}(\Omega)$ and $W^{1,\Theta}_{0}(\Omega)$ are separable and reflexive when $\Theta$ and $\tilde{\Theta}$ satisfy the $\Delta_{2}$-condition. 
	
	Recall also that $\Psi$ dominates $\Theta$ near infinity, in short we write $\Theta<\Psi$, if there exist positive constants $t_{0}$ and $k$ such that 
	\[
	\Theta(t)\leq \Psi(kt), \quad \mbox{for all} \hspace{0,2cm} t\geq t_{0}.
	\]
	If $\Theta<\Psi$ and $\Psi<\Theta$, then we say that $\Theta$ and $\Psi$ are equivalent, and we denote by $\Theta\approx \Psi$.  Let $\Theta_*$ be the inverse of the function
	$$
	t\in(0,\infty)\mapsto\int_0^t\frac{\Theta^{-1}(s)}{s^{\frac{N+1}{N}}}ds
	$$
	which extends to $\mathbb{R}$ by  $\Theta_*(t)=\Theta_*(-t)$ for  $t\leq 0$. We say that $\Theta$ increases essentially more slowly than $\Psi$ near infinity, in short we write $\Theta<<\Psi$, if and only if for every positive constant $k$ one has
	\[
	\lim_{t\rightarrow \infty}\frac{\Theta(kt)}{\Psi(t)}=0.
	\]
	It is important to emphasize that if $\Theta<\Psi<<\Theta_*$, then the following embedding
	$$
	\displaystyle W_0^{1,\Theta}(\Omega) \hookrightarrow L_\Psi(\Omega),
	$$
	is compact. In particular, since $\Theta<<\Theta_*$ (cf. \cite[Lemma 4.14]{Gz1}), we have that $W_0^{1,\Theta}(\Omega)$ is compactly embedded into $L_\Theta(\Omega)$.
	Furthermore, we have that $W_0^{1,\Theta}(\Omega)$ is continuous embedded into $L_{\Theta_*}(\Omega)$. Finally, we recall the following proposition due to N. Fukagai et al. \cite{Fuk_1} which can be written in the following way:
	\begin{proposition}\label{Narukawa}
		Assume that $(\phi_1)-(\phi_3)$ hold and set
		$$
		\zeta_0(t)=\min\{t^\ell,t^m\}~~\mbox{and}~~ \zeta_1(t)=\max\{t^\ell,t^m\},~~ t\geq 0.
		$$
		Then $\Phi$ satisfies the following estimates:
		$$
		\zeta_0(t)\Phi(\rho)\leq\Phi(\rho t)\leq \zeta_1(t)\Phi(\rho), \quad \rho, t> 0,
		$$
		$$
		\zeta_0(\|u\|_{\Phi})\leq\int_\Omega\Phi(u)dx\leq \zeta_1(\|u\|_{\Phi}), \quad u\in L_{\Phi}(\Omega).
		$$
	\end{proposition}
	As a consequence, by using Proposition \ref{Narukawa}, we can prove the following result:
	\begin{proposition}\label{emb-ell}
		Assume that $(\phi_1)-(\phi_4)$ hold. Then $L_\Phi(\Omega) \hookrightarrow L^\ell(\Omega)$ and $L_{\Phi_*}(\Omega) \hookrightarrow L^{\ell^*}(\Omega)$.
	\end{proposition}
	
	For the function $\Phi_*$ we obtain similar estimates given by the following result. 
	
	\begin{proposition}\label{lema_naru_*}
		Assume that  $\phi$ satisfies $(\phi_1)-(\phi_3)$.  Set
		$$
		\zeta_2(t)=\min\{t^{\ell^*},t^{m^*}\},~~ \zeta_3(t)=\max\{t^{\ell^*},t^{m^*}\},~~  t\geq 0
		$$
		where $1<\ell,m<N$ and $m^* = \frac{mN}{N-m}$, $\ell^* = \frac{\ell N}{N-\ell}$.  Then
		$$
		\ell^*\leq\frac{t\Phi'_*(t)}{\Phi_*(t)}\leq m^*,~t>0,
		$$
		$$
		\zeta_2(t)\Phi_*(\rho)\leq\Phi_*(\rho t)\leq \zeta_3(t)\Phi_*(\rho),~~ \rho, t> 0,
		$$
		$$
		\zeta_2(\|u\|_{\Phi_{*}})\leq\int_\Omega\Phi_{*}(u)dx\leq \zeta_3(\|u\|_{\Phi_*}),~ u\in L_{\Phi_*}(\Omega).
		$$
	\end{proposition}

	Now, we shall consider the following useful result
	\begin{proposition}\label{salvadora}
		Suppose that assumptions $(\phi_1)-(\phi_2)$ and $(\phi_4)$ are satisfied. 
		Then we obtain the following statements: 
		\begin{equation*}
			\ell \Phi(t) \leq \phi(t) t^2 \leq m \Phi(t), t \in \mathbb{R}.
		\end{equation*}
	\end{proposition}
	
	\section{The Nehari method for nonhomogeneous operators}
	In the present section we shall consider the Nehari method for the quasilinear operator driven by the $\Phi$-Laplacian operator. The main idea is to consider the fibering map $\gamma: [0, \infty) \to [0, \infty)$ given by $t \mapsto \gamma(t) = J_{\lambda}(tu)$ where $u \in W^{1,\Phi}_0(\Omega)$, see for instance \cite{Pohozaev, Pokhozhaev}. Furthermore, $\gamma$ is in $C^1$ class and $\gamma'(t) = 0$ if and only if $t u \in \mathcal{N}_{\lambda}$. In particular, $u \in \mathcal{N}_{\lambda}$ if and only if $\gamma'(1) = 0$. Under these conditions, we shall consider the function $\gamma$ in the following form 
	\begin{equation*}
		\gamma(t) = \int_{\Omega} \Phi( t |\nabla u|) dx - \dfrac{\lambda t^q}{q} \int_{\Omega} a(x) |u|^q dx - \dfrac{t^p}{p} \int_{\Omega} |u|^p dx, t \geq 0.
	\end{equation*}
	As a consequence, we also mention that 
	\begin{equation*}
		\gamma'(t) = \int_{\Omega} \phi( t |\nabla u|) t |\nabla u|^2 dx - \lambda t^{q-1} \int_{\Omega} a(x) |u|^q dx - t^{p-1} \int_{\Omega} |u|^p dx, t \geq 0
	\end{equation*}
	and
	\begin{equation*}
		\gamma''(t) = \int_{\Omega} \phi( t |\nabla u|)  |\nabla u|^2 + \phi'( t |\nabla u|)  t |\nabla u|^3 dx - \lambda (q-1) t^{q-2} \int_{\Omega} a(x) |u|^q dx - (p-1) t^{p-2} \int_{\Omega} |u|^p dx, t > 0.
	\end{equation*}
	Similarly, we also mention that 
	\begin{equation*}\label{ed1}
		J_{\lambda}'(u) u = \int_{\Omega} \phi(|\nabla u|) |\nabla u|^2 dx - \lambda \int_{\Omega} a(x) |u|^q dx -\int_{\Omega} |u|^p dx , u \in W^{1, \Phi}_0(\Omega).
	\end{equation*}
	and
	\begin{equation}\label{ed2}
		J_{\lambda}''(u) (u,u) = \int_{\Omega} 
		\phi(|\nabla u|) |\nabla u|^2 + \phi'(|\nabla u|)|\nabla u|^3 dx - \lambda(q-1) \int_{\Omega} a(x) |u|^q dx -(p -1)\int_{\Omega} |u|^p dx , u \in W^{1, \Phi}_0(\Omega).
	\end{equation}
	Under our assumptions we observe that 
	\begin{equation*}
		\gamma''(t) = \dfrac{J''_\lambda(tu)(tu,tu)}{t^2}, t > 0, u \in W^{1, \Phi}_0(\Omega).
	\end{equation*}
	It is worthwhile to mention that the functionals $u \mapsto \int_{\Omega} \Phi(|\nabla u|) dx$ and $u \mapsto \int_{\Omega} \phi(|\nabla u|) |\nabla u|^2 dx$ are not homogeneous. This is a difficulty in order to apply the Nehari method and the nonlinear Rayleigh quotient. This difficulty is overcome by proving that the function $t \mapsto R_n(t u)$ admits exactly one critical point. In fact, for a general $N$-function $\Phi$ the Nehari method can be applied using some fine estimates together with the fact that the norm in $W^{1,\Phi}_0(\Omega)$ is weakly lower semicontinuous, see \cite{Fig, ed2}. However, up to now, there is no result for a non-homogeneous operator driven by the $\Phi$-Laplacian taking into account the nonlinear Rayleigh quotient. More precisely, we consider the following result:
	
	\begin{proposition}\label{prop1}
		Suppose that assumptions $(\phi_1)-(\phi_4)$ and \ref{paper4f1}--\ref{paper4f3} are satisfied. Then the function $t \mapsto R_n(tu)$ has the following properties:
		\begin{itemize}
			\item[i)] There holds
			$$\lim_{t \to 0^+} \dfrac{R_n(tu)}{t^{m -q}}> 0, \lim_{t \to 0^+} \dfrac{\frac{d}{dt} R_n(tu)} {t^{m -q-1}} > 0.$$
			\item[ii)] There holds
			$$\lim_{t \to + \infty} \dfrac{R_n(tu)}{t^{p -q}} < 0, \lim_{t \to +\infty} \dfrac{\frac{d}{dt} R_n(tu)} {t^{p -q-1}} < 0.$$
		\end{itemize}
	\end{proposition} 
	
	\begin{proof}
		Firstly, we observe that 
		\begin{equation*}
			R_n(tu) = \dfrac{ t^{2 - q}\int_{\Omega} \phi(t |\nabla u |) |\nabla u|^2 dx - t^{p-q} \| u\|_p^p}{\|u\|^q_{q,a}}, t > 0, u \in W^{1,\Phi}_0(\Omega).
		\end{equation*}
		Moreover, we mention that 
		\begin{equation*}\label{importante}
			\dfrac{d}{dt} R_n(tu) = \dfrac{ (2 - q)t^{1 - q}\int_{\Omega} \phi(t |\nabla u |) |\nabla u|^2 dx + t^{2-q}\int_{\Omega} \phi'(t |\nabla u |) |\nabla u|^3 dx - (p-q) t^{p-q-1} \| u\|_p^p}{\|u\|^q_{q,a}}, t > 0.
		\end{equation*}
		Hence, we obtain 
		\begin{equation*}
			\dfrac{R_n(tu)}{t^{m-q}} = \dfrac{ t^{2 - m}\int_{\Omega} \phi(t |\nabla u |) |\nabla u|^2 dx - t^{p-m} \| u\|_p^p}{\|u\|^q_{q,a}}.
		\end{equation*}
		Therefore, 
		\begin{equation}\label{e0}
			\lim_{t \to 0} \dfrac{R_n(tu)}{t^{m-q}} = \lim_{t \to 0} \left(\dfrac{ t^{2 - m}\int_{\Omega} \phi(t |\nabla u |) |\nabla u|^2 dx}{\|u\|^q_{q,a}} \right).
		\end{equation}
		The last assertion together with \eqref{e0} and Proposition \ref{salvadora} imply that 
		\begin{equation}\label{e1}
			\lim_{t \to 0} \dfrac{R_n(tu)}{t^{m-q}} \geq \lim_{t \to 0} \left(\dfrac{ t^{- m}\int_{\Omega} \Phi( t |\nabla u |)  dx}{\|u\|^q_{q,a}} \right).
		\end{equation}
		In light of \eqref{e1} and Proposition \ref{Narukawa} we see that 
		\begin{equation*}
			\lim_{t \to 0} \dfrac{R_n(tu)}{t^{m-q}} \geq \dfrac{ \int_{\Omega} \Phi(  |\nabla u |) dx}{\|u\|^q_{q,a}}  > 0.
		\end{equation*}
		Similarly, taking into account hypothesis $(\phi_4)$, we obtain that 
		\begin{equation*}
			\lim_{t \to 0} \dfrac{1} {t^{m-q-1}}\dfrac{d}{d t} R_n(tu) \geq  \dfrac{ \int_{\Omega} \Phi(|\nabla u |)  dx}{\|u\|^q_{q,a}} > 0.
		\end{equation*}
		As a consequence, the proof of item $i)$ follows. 
		
		Analogously, we prove that 
		\begin{equation*}
			\lim_{t \to \infty} \dfrac{R_n(tu)}{t^{p-q}} = - \dfrac{\| u\|_p^p}{\|u\|^q_{q,a}} + \lim_{t \to \infty} \left(\dfrac{ t^{2 - p}\int_{\Omega} \phi(t |\nabla u |) |\nabla u|^2 dx}{\|u\|^q_{q,a}} \right).
		\end{equation*}
		Now, by using Proposition \ref{salvadora} and Proposition \ref{Narukawa}, we obtain that 
		\begin{equation*}
			\lim_{t \to \infty} \dfrac{R_n(tu)}{t^{p-q}} = - \dfrac{\| u\|_p^p}{\|u\|^q_{q,a}} + \lim_{t \to \infty} \dfrac{  t^{m - p}\int_{\Omega} \Phi(|\nabla u |)  dx}{\|u\|^q_{q,a}} =  - \dfrac{\| u\|_p^p}{\|u\|^q_{q,a}} < 0.
		\end{equation*}
		Furthermore, by using the same ideas discussed just above, we obtain that 
		\begin{equation*}
			\lim_{t \to \infty} \dfrac{1} {t^{p-q-1}} \dfrac{d}{d t} R_n(tu)  \leq - (p -q) \dfrac{\| u\|_p^p}{\|u\|^q_{q,a}} < 0.
		\end{equation*}
		This ends the proof. 	
	\end{proof}
	
	\begin{proposition}\label{propimpor}
		Suppose that assumptions $(\phi_1)-(\phi_4)$ and \ref{paper4f1}--\ref{paper4f4} are satisfied. Then for each $u \in  W^{1,\Phi}_0(\Omega) \setminus \{0\}$ there exists an unique $t_n(u) > 0$ such that 
		\begin{equation}\label{e5}
			\dfrac{d}{dt} R_n(tu) = 0 \,\, \mbox{for} \,\, t = t_n(u). 
		\end{equation}
		Furthermore,  $t_n:W_0^{1,\Phi}(\Omega)\setminus\{0\}\to (0,\infty)$ is a $C^1$-functional.
	\end{proposition} 
	\begin{proof}
		According to Proposition \ref{prop1} we infer that there exists at least one number $t_n(u) > 0$ such that \eqref{e5} is satisfied. On the other hand, we observe that $$\dfrac{d} {dt} R_n(tu) = 0, t > 0$$ is equivalent to the following identity
		\begin{equation*}
			(p-q) \| u\|_p^p = \int_{\Omega} \dfrac{[(2 - q) \phi(t |\nabla u |) + \phi'(t |\nabla u|) |t \nabla u| ] |\nabla u|^2}{t^{p-2}} dx.
		\end{equation*}
		The last inequality can be rewritten as follows 
		\begin{equation}\label{salvou0}
			(p-q) \| u\|_p^p = \int_{\Omega} \dfrac{[(2 - q) \phi(t |\nabla u |) + \phi'(t |\nabla u|) |t \nabla u| ] |\nabla u|^p}{|t \nabla u|^{p-2}} dx.
		\end{equation}
		Now, by using hypothesis $(H_3)$, we obtain that the function $h:(0, \infty) \to \mathbb{R}$ given by 
		\begin{equation*}
			h(t) = \int_{\Omega} \dfrac{[(2 - q) \phi(t |\nabla u |) + \phi'(t |\nabla u|) |t \nabla u| ] }{|t \nabla u|^{p-2}}|\nabla u|^pdx
		\end{equation*}
		is strictly decreasing. Furthermore, by using assumption $(\phi_4)$, we have that $$\displaystyle\lim_{t\to 0}h(t)=\infty, \, \, \mbox{and} \,\, \displaystyle\lim_{t\to \infty}h(t)=0. $$ Hence, the identity \eqref{salvou0} admits at most one solution $t_n(u) > 0$ for each $u \in W^{1,\Phi}_0(\Omega) \setminus \{0\}$. Therefore, there exists a unique $t_n(u) > 0$ such that \eqref{e5} is verified.
		
		Now, we shall prove that $t_n:W_0^{1,\Phi}(\Omega)\setminus\{0\}\to (0,\infty)$ is a $C^1$-functional. Let us fix $u_0\in W_0^{1,\Phi}(\Omega)\setminus\{0\}$. At this stage, we know that $t_n(u_0)$ is well defined. Consider the auxiliary function $F:  W_0^{1,\Phi}(\Omega)\setminus\{0\}\times(0,\infty)\to \mathbb{R}$ given by $F(v,t)=R_n'(tv)tv$. Recall also that
		$$F_t(v,t)\big|_{(v,t)=(u_0,t_n(u_0))}=\frac{R_n''(t_n(u_0)u_0)(t_n(u_0)u_0)^2}{t_n(u_0)}<0.$$
		It follows from Implicit Function Theorem \cite{drabek} that there exists an open set $U\ni u_o$ and a $C^1$-functional $t:U\to (0,\infty)$ such that 
		$$F(v,t(v))=R_n'(t(v)v)t(v)v=0,~\forall v\in U.$$
		Furthermore, for each $v\in W_0^{1,\Phi}(\Omega)\setminus\{0\}$, there exists an unique $t_n(v)$ such that $R_n'(t_n(v)v)t_n(v)v=0$. Therefore, $t(v)=t_n(v)$ holds for each $v\in U$. The last assertion implies that $t_n\in C^1(U;(0,\infty))$. Since $u_0$ is an arbitrary function we conclude that $t_n\in C^1(W^{1,\Phi}_0(\Omega) \setminus \{0\};(0,\infty))$. This ends the proof. 
	\end{proof}

	\begin{remark}\label{rmk1}
		Let $t > 0$ and $u \in W^{1,\Phi}_0(\Omega) \setminus \{0\}$ be fixed. It is important to mention that $t u \in \mathcal{N}_{\lambda}$ if and only if $R_n(tu) = \lambda$. More specifically, by using  \eqref{b11}, we deduce that $R_n(t u)= \lambda$ if and only if $J_{\lambda}'(t u) t u = 0$. Furthermore, we mention that 
		$R_n(tu) > \lambda$ if and only if $J_{\lambda}'(tu) tu > 0$ and
		$R_n(tu) < \lambda$ if and only if $J_{\lambda}'(tu) tu < 0$.
	\end{remark}
	
	\begin{remark}\label{rmk11} Let $u\in W^{1,\Phi}_0(\Omega) \setminus\{0\}$ and $t > 0$ be fixed. Now, taking into account \eqref{b111}, we infer that ${R}_e(t u)=\lambda$ if and only if $J_{\lambda}(tu)=0$. Furthermore, we obtain that  
		$R_e(tu) > \lambda$ if and only if $J_{\lambda}(tu) > 0$ and $R_e(tu) < \lambda$ if and only if $J_{\lambda}(tu) < 0$.
	\end{remark}
	
	It is important to stress that the function $Q_n: (0, \infty) \to \mathbb{R}$ given by $Q_n(t) = R_n(tu)$ satisfies $Q_n'(t) > 0$ if and only if $t \in (0, t_n(u))$. Moreover, $Q_n'(t) < 0$ if and only if $t \in (t_n(u), \infty)$ and $Q_n'(t) = 0$ if and only if $t = t_n(u)$.
	At this stage, we shall consider the auxiliary functional $\Lambda_n :  W^{1,\Phi}_0(\Omega) \setminus \{0\} \to \mathbb{R}$ given by 
	\begin{equation*}
		\Lambda_n(u) = \sup_{t > 0} R_n(t u) = R_n(t_n(u) u).
	\end{equation*}
	Therefore, we can consider the following extreme 
	\begin{equation*}
		\lambda^* = \inf_{u \in  W^{1,\Phi}_0(\Omega) \setminus \{0\}} \Lambda_n(u).
	\end{equation*}
	As a product, we shall consider the following result:
	
	\begin{proposition}\label{marcos}
		Suppose that assumptions $(\phi_1)-(\phi_4)$ and \ref{paper4f1}--\ref{paper4f4} are satisfied. Then the functional $u \mapsto \Lambda_n(u)$ satisfies the following properties: 
		\begin{itemize}
			\item[i)] $\Lambda_n$ is zero homogeneous, that is, $\Lambda_n(s u) = \Lambda_n(u)$ for each $s > 0$ and $u \in W^{1,\Phi}_0(\Omega) \setminus \{0\}$;
			\item[ii)] The functional $\Lambda_n:W_0^{1,\Phi}(\Omega)\setminus\{0\}\to (0,\infty)$ is differentiable and weakly lower semicontinuous;
			\item[iii)] There exists $c > 0$ such that $\|t_n(u)u\|\geq c$ for each $u \in W^{1,\Phi}_0(\Omega) \setminus \{0\}$;
			\item[iv)] There exists $C=C(\ell,m,p,q,\Omega) > 0$ such that $\Lambda_n(u) \geq C$ for each $u \in W^{1,\Phi}_0(\Omega) \setminus \{0\}$;
			\item[v)] $\lambda^* > 0$ is attained, i.e., there exists $u^*\in W^{1,\Phi}_0(\Omega) \setminus \{0\}$ such that $\lambda^*=\Lambda_n(u^*)$;
			\item[vi)] $v^*:=t_n(u^*)u^*$ is a weak solution of the problem
			\begin{equation}\label{pia}
				\left\{
				\begin{array}{lr}
					-2\Delta_\Phi u -\dive(\phi'(|\nabla u|)|\nabla u|\nabla u)= q\lambda^* a(x) |u|^{q-2}u + p|u|^{p-2}u, & x\in\Omega,\\
					u = 0, & x \in \partial \Omega.
				\end{array}
				\right. \tag{$P_\lambda'$}
			\end{equation}
		\end{itemize}
	\end{proposition} 
	\begin{proof}
		Firstly, we shall prove the item $i)$. Let $s>0$ and $u\in W^{1,\Phi}_0(\Omega) \setminus \{0\}$ be fixed. Recall that $$R_n'\left(\frac{t_n(u)}{s}(su)\right)\left(\frac{t_n(u)}{s}(su)\right)=0.$$
		As a consequence, we obtain that 
		$t_n(su)=t_n(u)/s$. Therefore, 
		$$\Lambda_n(su)=R_n(t_n(su)su)=R_n\left(\frac{t_n(u)}{s}su\right)=R_n(t_n(u)u)=\Lambda_n(u).$$
		
		Now we shall prove the item $ii)$. Initially, by using Proposition \ref{propimpor}, we observe that $\Lambda_n\in C^1(W^{1,\Phi}_0(\Omega) \setminus \{0\};(0,\infty))$. Now, we consider a sequence $(u_k)\subset W_0^{1,\Phi}(\Omega)$ such that $u_k\rightharpoonup u\neq 0$ in $ W_0^{1,\Phi}(\Omega)$. 
		Since $v\mapsto R_n'(tv)(tv)$ is weakly lower semicontinuous we infer that
		$$0=R_n'(t_n(u)u)(t_n(u)u)\leq R_n'(t_n(u)u_k)(t_n(u)u_k), t > 0, ~k>>1.$$
		Consequently, $t_n(u)\leq t_n(u_k)$. Under these conditions, taking into account that $v\mapsto R_n(tv)$ is also weakly lower semicontinuous, we obtain
		$$\Lambda_n(u)=R_n(t_n(u)u)\leq \liminf_{k\to\infty} R_n(t_n(u)u_k)\leq \liminf_{k\to\infty} R_n(t_n(u_k)u_k)=\liminf_{k\to\infty}\Lambda_n(u_k), t > 0.$$
		This ends the proof of item $ii)$.

		Now, we shall prove the item $iii)$. As a first step, by using hypothesis $(\phi_4)$, we mention that
		\begin{eqnarray}
			0&=&\|t_n(u)u\|_q^q t_n(u)R_n'({t_n(u)}u)({t_n(u)}u)\nonumber\\
			&=&\int_\Omega \left[(2-q)\phi(t_n(u)|\nabla u|)+\phi'(|t_n(u)\nabla u|)t_n(u)|\nabla u|\right]|t_n(u)\nabla u|^2dx-(p-q)\|t_n(u)u\|_p^p\label{esti-t_n0}\\
			&\geq& (\ell-q)\int_\Omega \phi(t_n(u)|\nabla u|)|t_n(u)\nabla u|^2dx -(p-q)\|t_n(u)u\|_p^p.\label{esti-t_n}
		\end{eqnarray}
		Now, by using \eqref{esti-t_n} together with the embedding $W_0^{\Phi}(\Omega)\hookrightarrow L^p(\Omega)$ and Proposition \ref{Narukawa}, we infer that
		$$\|t_n(u)u\|^p\geq\ell\frac{\ell-q}{S_p(p-q)}\min\left\{\|t_n(u)u\|^\ell,\|t_n(u)u\|^m\right\},~\forall u\in W_0^{1,\Phi}(\Omega)\setminus\{0\}.$$
		Hence, the proof of item $iii)$ follows from the previous estimates together with hypothesis \ref{paper4f1}.
		
		Now, we shall prove $iv)$. It follows from \eqref{esti-t_n0} and $(\phi_4)$ that 
		\begin{eqnarray}
			\frac{m-q}{p-q}\int_\Omega \phi(t_n(u)|\nabla u|)|t_n(u)\nabla u|^2dx \geq \|t_n(u)u\|_p^p.\nonumber
		\end{eqnarray}
		Furthermore, by using Proposition \ref{Narukawa},  we deduce
		\begin{eqnarray}\label{est-ln}
			\Lambda_n(u)=R_n(t_n(u)u)&=& \frac{\displaystyle\int_{\Omega} \phi(|t_n(u)\nabla u|) |t_n(u)\nabla u|^2 dx- \|t_n(u)u\|_p^p}{\|t_n(u)u\|^q_{q,a}}\nonumber\\
			&\geq & \frac{\displaystyle\frac{p-m}{p-q}\int_{\Omega} \phi(|t_n(u)\nabla u|) |t_n(u)\nabla u|^2 dx}{\|t_n(u)u\|^q_{q,a}}\nonumber\\
			&\geq &	\ell\frac{p-m}{p-q}\frac{\min\left\{\|t_n(u)u\|^\ell,\|t_n(u)u\|^m\right\}}{\|t_n(u)u\|^q_{q,a}}
		\end{eqnarray}
		Now, by using \eqref{est-ln}, H\"older's inequality and the embedding $W_0^{1,\Phi}(\Omega)\hookrightarrow L^{\ell}(\Omega)$ and $ii)$, we infer that
		\begin{eqnarray}\label{lambda*}
			\Lambda_n(u)\geq \frac{\ell (p-m)}{S_\ell^q(p-q)\|a\|_{\frac{\ell}{\ell-q}}}\min\left\{\|t_n(u)u\|^{\ell-q},\|t_n(u)u\|^{m-q}\right\}\geq \frac{\ell (p-m)}{S_\ell^q(p-q)\|a\|_{\frac{\ell}{\ell-q}}}\min\left\{c^{\ell-q},c^{m-q}\right\}=:C.
		\end{eqnarray}
		At this stage, we consider a minimizer sequence $(u_k)$ for $\lambda^*$. Notice that $\Lambda_n$ is 0-homogeneous. Hence, we assume that $t_n(u_k)=1$. It follows from \eqref{lambda*} that 
		$$ \lambda^*+1\geq\lambda_n(u_k)\geq \frac{\ell (p-m)}{S_\ell^q(p-q)\|a\|_{\frac{\ell}{\ell-q}}}\min\left\{\|u_k\|^{\ell-q},\|u_k\|^{m-q}\right\},~k>>1.$$
		Therefore, the sequence $(u_k)$ is bounded in $W^{1,\Phi}_0(\Omega)$. Up to a subsequence, there exist $u \in W^{1,\Phi}_0(\Omega)$ such that $u_k\rightharpoonup u$. Now, we claim that $u\neq 0$ is satisfied. The proof of this claim follows arguing by contradiction. Let us assume that $u\equiv 0$. According to \eqref{esti-t_n}, $(\phi_4)$ and Proposition \ref{Narukawa} we obtain
		$$(p-q)\|u_k\|_p^p\geq(\ell-q)\int_\Omega \phi(|\nabla u_k|)|\nabla u_k|^2dx\geq \ell\int_\Omega \Phi(|\nabla u_k|)dx\geq \ell\min\left\{\|u_k\|^{\ell},\|u_k\|^{m}\right\}.$$
		The last assertion implies that $u_k\to 0$ in $W_0^{1,\Phi}(\Omega)$ which does not make sense due to item $ii)$. Hence, $u\not\equiv 0$ holds. Furthermore, by using the fact that $\Lambda_n$  is weak lower semicontinuous, we mention that
		$$\lambda^*\leq \Lambda_n(u)\leq \liminf_{k\to\infty}\Lambda_n(u_k)=\lambda^*.$$
		This ends the proof of item $v)$.
		
		Now, we shall prove the item $vi)$. Firstly, we consider the auxiliary functions $F, G :  W^{1,\Phi}_0(\Omega) \to \mathbb{R}$ given by
		$$F(u):= \int_{\Omega} \phi(|\nabla v^*|) |\nabla v^*|^2 dx- \|v^*\|_p^p\qquad\mbox{and}\qquad G(u):=\|v^*\|^q_{q,a}.$$
		Recall that $v^*$ is a critical point of $\Lambda_n$ and $\lambda^*=\frac{F(v^*)}{G(v^*)}$. Therefore,
		\begin{eqnarray}
			0&=& \frac{1}{G(v^*)}\left[F'(v^*)\varphi-\frac{F(v^*)}{G(v^*)}G'(v^*)\varphi\right]=\frac{1}{G(v^*)}\left[F'(v^*)\varphi-\lambda^*G'(v^*)\varphi\right]\nonumber\\
			&=& \frac{1}{G(v^*)} \left\{\int_{\Omega} \left[2\phi(|\nabla v^*|)+\phi'(|\nabla v^*|)|\nabla v^*|\right]\nabla v^*\nabla \varphi dx  -p\int_{\Omega} |v^*|^{p-2}v^*\varphi dx - q\lambda^* \int_{\Omega} a(x) |v^*|^{q-2}v^*\varphi dx\right\} \nonumber
		\end{eqnarray}
		holds for each $\varphi \in W^{1, \Phi}_0(\Omega)$.
		This ends the proof.	
	\end{proof}

	\begin{remark}
		Here we observe that the operator in Problem \eqref{pia} can be written as follows
		$$\Delta_{\Phi_1}u=2\Delta_\Phi u+\dive(\phi'(|\nabla u|)|\nabla u|\nabla u),~u\in W_0^{1,\Phi}(\Omega)$$
		where $\Phi_1(t)=\phi(t)t^2$. In fact, we mention that $\frac{\Phi_1'(t)}{t}=\phi'(t)t+2\phi(t), t > 0$.
	\end{remark}	
	
	Now, using the functional $R_e$ instead of $R_n$, we can consider the following auxiliary result 
	\begin{proposition}\label{ele}
		Suppose that assumptions $(\phi_1)-(\phi_4)$ and \ref{paper4f1}--\ref{paper4f3} are satisfied. Then the function $t \mapsto R_e(tu)$ has the following properties:
		\begin{itemize}
			\item[i)] There holds
			$$\lim_{t \to 0^+} \dfrac{R_e(tu)}{t^{m -q}}> 0, \lim_{t \to 0^+} \dfrac{\frac{d}{dt} R_e(tu)} {t^{m -q-1}} > 0.$$
			\item[ii)] There holds
			$$\lim_{t \to + \infty} \dfrac{R_e(tu)}{t^{p -q}} < 0, \lim_{t \to +\infty} \dfrac{\frac{d}{dt} R_e(tu)} {t^{p -q-1}} < 0.$$
		\end{itemize}
	\end{proposition}
	\begin{proof}
		The proof of this result follows using the same ideas discussed in the proof Proposition \ref{propimpor}. We omit the details. 
	\end{proof}
	
	\begin{proposition}\label{puloug}
		Suppose that assumptions $(\phi_1)-(\phi_4)$ and \ref{paper4f4} are satisfied. Then the function $$t \mapsto \dfrac{\phi(t)t^2 - q \Phi(t)}{t^p}, t > 0,$$ 
		is strictly decreasing. 
	\end{proposition}
	\begin{proof}
		Let $\xi_1, \xi_2 \in (0, \infty), \xi_1 > \xi_2$ be fixed. According to hypothesis \ref{paper4f4} it follows that 
		\begin{equation*}
			\dfrac{(2-q)\phi(\xi_1) + \phi'(\xi_1)\xi_1}{\xi_1^{p-2}} < \dfrac{(2-q)\phi(\xi_2) + \phi'(\xi_2)\xi_2}{\xi_2^{p-2}}.
		\end{equation*}
		The last estimate implies that 
		\begin{equation*}
			[(2-q)\phi(\xi_1)\xi_1 + \phi'(\xi_1)\xi_1^2] \xi_2^{p-1} < [(2-q)\phi(\xi_2)\xi_2 + \phi'(\xi_2)\xi_2^2] \xi_1^{p-1}.
		\end{equation*}
		Integrating the last expression in the variable $\xi_1$ over the interval $[0,t]$ and using the integration by parts we infer that
		\begin{equation*}
			[-q\Phi(t) + \phi(t)t^2] \xi_2^{p-1} < [(2-q)\phi(\xi_2)\xi_2 + \phi'(\xi_2)\xi_2^2] (t^{p}/p).
		\end{equation*}
		Now, integrating over the interval $[0,s]$ in the variable $\xi_2$, we also obtain that 
		\begin{equation*}
			[-q\Phi(t) + \phi(t)t^2] (s^{p}/p) < [-q\Phi(s) + \phi(s)s^2] (t^{p}/p), t \geq s \geq 0.
		\end{equation*}
		The desired follows taking $t > s > 0$. This ends the proof.
	\end{proof}
	
	\begin{proposition}\label{propimpor2}
		Suppose that assumptions $(\phi_1)-(\phi_4)$ and \ref{paper4f1}--\ref{paper4f4} are satisfied. Then for each $u \in  W^{1,\Phi}_0(\Omega) \setminus \{0\}$ there exists an unique $t_e(u) > 0$ such that 
		\begin{equation*}\label{e56}
			\dfrac{d}{dt} R_e(tu) = 0 \,\, \mbox{for} \,\, t = t_e(u).
		\end{equation*}
		Furthermore,  $t_e:W_0^{1,\Phi}(\Omega)\setminus\{0\}\to (0,\infty)$ is a $C^1$-functional.
	\end{proposition}
	\begin{proof}
		Recall that 
		\begin{equation*}
			R_e(tu) = \dfrac{ t^{-q}\int_{\Omega} \Phi(t |\nabla u |) dx - (t^{p-q}/p) \| u\|_p^p}{\|u\|^q_{q,a}/q}, t > 0, u \in W^{1,\Phi}_0(\Omega).
		\end{equation*}
		As a consequence, we mention that 
		\begin{equation*}
			\dfrac{d}{dt} R_e(tu) = \dfrac{ - q t^{ -1-q}\int_{\Omega} \Phi(t |\nabla u |) dx + t^{-1-q}\int_{\Omega} \phi(t |\nabla u |) |t\nabla u|^2 dx - (p-q)t^{p-q-1}\| u\|_p^p/p}{\|u\|^q_{q,a}/q}, t > 0.
		\end{equation*}	
		Therefore, 
		\begin{equation}\label{eq}
			\dfrac{d}{dt} R_e(tu) = 0, t > 0,
		\end{equation}
		is equivalent to the following identity
		\begin{equation*}\label{salvou}
			\dfrac{(p-q)}{p} \| u\|_p^p = \int_{\Omega} \dfrac{[\phi(t |\nabla u |)|t \nabla u|^2 - q \Phi(t |\nabla u|)  ] |\nabla u|^p}{|t \nabla u|^{p}} dx.
		\end{equation*}
		In view of Proposition \ref{puloug} it follows that the function $L: (0, \infty) \to \mathbb{R}$ given by 
		\begin{equation*}
			L(t) = \int_{\Omega} \dfrac{[\phi(t |\nabla u |)|t \nabla u|^2 - q \Phi(t |\nabla u|)  ] |\nabla u|^p}{|t \nabla u|^{p}} dx
		\end{equation*}
		is strictly decreasing. Moreover, we also mention that $$\displaystyle \lim_{t\to 0}L(t)=\infty \,\, \mbox{and} \,\, \displaystyle \lim_{t\to \infty}L(t)=0. $$ Hence, the identity \eqref{eq} admits at most a root $t_e(u) > 0$ for each $u \in W^{1,\Phi}_0(\Omega) \setminus \{0\}$.
		According to Proposition \ref{ele} we obtain that $t_e(u) > 0$ is unique. This ends the proof.
	\end{proof}
	
	\begin{proposition}\label{2}
		Suppose that assumptions $(\phi_1)-(\phi_4)$ and \ref{paper4f1}--\ref{paper4f4} are satisfied. Assume also that $u\in W^{1,\Phi}_0(\Omega) \setminus \{0\}$ satisfies $R_n(tu)=\lambda$ for some $t>0$. Then $R^\prime_n(tu)>0$ if and only if $J_{\lambda}^{\prime\prime}(tu)(tu,tu)>0$. Furthermore, we obtain that $R^\prime_n(tu)<0$ if and only if $J_{\lambda}^{\prime\prime}(tu)(tu,tu)<0$ and $R^\prime_n(tu)=0$ if and only if $J_{\lambda}^{\prime\prime}(tu)(tu,tu)=0$.
	\end{proposition}
	\begin{proof}
		Firstly, we observe that 		
		\begin{equation}\label{rel-J-R}		
			R^\prime_n(tu)u=\frac{d}{dt}R_n(tu)=\frac{1}{t}\frac{J_{\lambda}^{\prime\prime}(tu)(tu,tu)}{\|tu\|_{q,a}^q}
		\end{equation}
		holds for each $u \in W^{1,\Phi}_0(\Omega) \setminus \{0\}$ such that $R_n(tu) = \lambda, t > 0$. The desired results follow using the last identity.  
	\end{proof}

	\begin{proposition}\label{1}
		Suppose that assumptions $(\phi_1)-(\phi_4)$ and \ref{paper4f1}--\ref{paper4f4} are satisfied. Assume also that $u\in W^{1,\Phi}_0(\Omega) \setminus \{0\}$ satisfies $R_e (tu) =\lambda$ for some $t > 0$. Then we obtain $R^\prime_e (tu) > 0$ if and only if $J_{\lambda}^\prime
		(tu)tu > 0$. Furthermore, we obtain $R^\prime_e(tu)<0$ if and only if $J_{\lambda}^\prime(tu)tu<0$. 
		Moreover, we mention also that $R^\prime_e(tu)=0$ if, and only if $J_{\lambda}^\prime(tu)tu=0$.
	\end{proposition}
	\begin{proof}
		Firstly, we mention that
		$$R^\prime_e(tu)u=\frac{d}{dt}R_e(tu)=\frac{q}{t}\frac{J_{\lambda}^\prime(tu)tu}
		{\displaystyle \|tu\|_{q,a}^q}$$
		holds for each $u \in W^{1,\Phi}_0(\Omega) \setminus \{0\}$ such that $R_e(tu) = \lambda, t > 0$. The desired results follow using the last identity.
	\end{proof}
	
	It is important to mention that the function $Q_e: (0, \infty) \to \mathbb{R}$ given by $Q_e(t) = R_e(tu)$ satisfies $Q_e'(t) > 0$ if and only if $t \in (0, t_e(u))$. Furthermore, $Q_e'(t) < 0$ if and only if $t \in (t_e(u), \infty)$ and $Q_e'(t) = 0$ if and only if $t = t_e(u)$.
	Now, we shall consider the functional $\Lambda_e :  W^{1,\Phi}_0(\Omega) \setminus \{0\} \to \mathbb{R}$ given by 
	\begin{equation*}
		\Lambda_e(u) = \sup_{t > 0} R_e(t u) = R_e(t_e(u) u).
	\end{equation*}
	Hence, we consider the following extreme 
	\begin{equation*}
		\lambda_* = \inf_{u \in  W^{1,\Phi}_0(\Omega) \setminus \{0\}} \Lambda_e(u).
	\end{equation*}
	As a consequence, we shall consider the following result:
	\begin{proposition}\label{marcos2}
		Suppose that assumptions $(\phi_1)-(\phi_4)$ and \ref{paper4f1}--\ref{paper4f4} are satisfied. Then the functional $u \mapsto \Lambda_e(u)$ satisfies the following statements: 
		\begin{itemize}
			\item[i)] $\Lambda_e$ is zero homogeneous, that is, $\Lambda_e(s u) = \Lambda_e(u)$ for each $s > 0$ and $u \in W^{1,\Phi}_0(\Omega) \setminus \{0\}$;
			\item[ii)] The functional $\Lambda_e: W_0^{1,\Phi}(\Omega)\to (0,\infty)$ is differentiable and weakly lower semicontinuous;
			\item[iii)] There exists $c > 0$ such that $\|t_e(u)u\|\geq c$ holds for each $u \in W^{1,\Phi}_0(\Omega) \setminus \{0\}$;
			\item[iv)] There exists $C=C(\ell,m,p,q,\Omega) > 0$ such that $\Lambda_e(u) \geq C$ for each $u \in W^{1,\Phi}_0(\Omega) \setminus \{0\}$;
			
			\item[v)] $\lambda_* > 0$ is attained, i.e., there exists $\widetilde{u}\in W_0^{1,\Phi}(\Omega)\setminus\{0\}$ such that $\lambda_* = \Lambda_e(\widetilde{u})$
			\item[vi)] It holds that the function $w^*:=t_e(\widetilde{u})\widetilde{u}$ is a solution of $(P_{\lambda_*})$.
		\end{itemize}
	\end{proposition} 
	\begin{proof}
		The proof follows the same lines discussed in the prof of Proposition \ref{marcos}. We omit the details.
	\end{proof}
	
	\begin{proposition}\label{proje}
		Suppose that assumptions $(\phi_1)-(\phi_4)$ and \ref{paper4f1}--\ref{paper4f4} are satisfied. Assume also that $\lambda \in (0, \lambda^*)$. Then the fibering map function $\gamma(t) = J_\lambda (tu), t > 0,  u \in W^{1, \Phi}_0(\Omega) \setminus \{0\}$ has the following properties:
		\begin{itemize}
			\item[i)] There exist unique $t^{n,+}(u) < t_n(u) < t^{n,-}(u)$ in such way that $t^{n,+}(u) u \in \mathcal{N}_{\lambda}^+$ and $t^{n,-}(u) u \in \mathcal{N}_{\lambda}^-$. Furthermore, $t = t^{n,+}(u)$ is a local minimum point for the fibering map $\gamma$.
			\item[ii)] It holds that $t = t^{n,-}(u)$ is a global maximum point for the function $\gamma$ for each $\lambda \in (0, \lambda_*)$. For $\lambda \in [\lambda_*, \lambda^*)$ we obtain that $t = t^{n,-}(u)$ is only a local maximum point for $\gamma$.
			\item[iii)] The functions $u \mapsto t^{n,+}(u)$ and $u \mapsto t^{n,-}(u)$ are in $C^1$ class.  
		\end{itemize}	
	\end{proposition}
	\begin{proof}
		Firstly, we shall prove the item $i)$. It is easy to verify that $tu \in \mathcal{N}_\lambda$ if and only if $R_n(tu) = \lambda$ where $t > 0$ and $u \in  (W^{1,\Phi}_0(\Omega) \setminus \{0\})$. 
		Recall that $\lambda \in (0,\lambda^*)$ and $u \in (W^{1,\Phi}_0(\Omega) \setminus \{0\})$. As a consequence, we obtain that
		\begin{equation*}
			0 < \lambda <  \lambda^* = \inf_{w \in  (W^{1,\Phi}_0(\Omega) \setminus \{0\})} \Lambda_n(w) \leq \Lambda_n(u).
		\end{equation*}
		Hence, the equation $R_n(tu) = \lambda$ has exactly two roots which are denoted by $t^{n,\pm}(u)$ such that $t^{n,+}(u)u \in \mathcal{N}_\lambda^+$ and $t^{n,-}(u)u \in \mathcal{N}_\lambda^-$. As a product, for each $u \in  W^{1,\Phi}_0(\Omega) \setminus \{0\}$, we obtain $t^{n,+}(u) < t_n(u) < t^{n,+}(u)$. It is important to mention also that 
		\begin{equation*}\label{al}
			\left.\dfrac{d}{d t} R_n(tu)\right|_{t=t_{n,+} (u)} > 0\qquad\mbox{and}\qquad \left.\dfrac{d}{d t} R_n(tu)\right|_{t=t_{n,-} (u)} < 0.
		\end{equation*} 
		As a consequence, by using Proposition \ref{2}, we infer that 
		\begin{equation*}
			\left.\dfrac{d^2}{d t^2} J_\lambda (tu)\right|_{t=t_{n,+} (u)} > 0\qquad\mbox{and}\qquad \left. \dfrac{d^2}{d t^2} J_\lambda(tu)\right|_{t=t_{n,-} (u)} < 0.
		\end{equation*}
		Hence, $t^{n,+}(u)$ is a local minimum point for the fibering map function $\gamma$. This ends the proof of item $i)$.
		
		Now we shall prove the item $ii)$. Using the same ideas discussed just above we observe that $t^{n,-}(u)$ is always a local maximum point for the fibering map $\gamma$ for each $\lambda \in (0, \lambda^*)$.  Furthermore, for each $\lambda \in (0, \lambda_*)$, we obtain that 
		\begin{equation*}
			R_n(t^{n,-}(u)u) = \lambda < \lambda_* \leq R_e(t^{n,-}(u)u).
		\end{equation*}
		In particular, by using Remark \ref{rmk11}, we obtain that 
		\begin{equation*}
			\gamma(t^{n,-}(u)) = J_\lambda (t^{n,-}(u)u) > J_{R_e(t^{n,-}(u)u)} (t^{n,-}(u)u) =0.
		\end{equation*}
		Hence, $t = t^{n,-}(u)$ is a global maximum point for the fibering map for each $\lambda \in (0, \lambda_*)$. This ends the proof of item $ii)$.
		
		Now, we shall prove the item $iii)$. Consider the auxiliary function $H^{\pm} : (W^{1,\Phi}_0(\Omega) \setminus \{0\}) \times (0, \infty) \to (-\infty, \infty)$ given by $H(u,t) = J'_{\lambda}(tu) tu$. Notice also that $H(u,t) = 0$ if and only if $t u \in \mathcal{N}_\lambda$. Furthermore, $\partial_t H^{\pm}(u,t) \neq 0$ if and only if $tu \in \mathcal{N}_\lambda^{\pm}$. In particular, for each $u \in \mathcal{N}_\lambda^{\pm}$ it follows that $\partial_t H^{\pm}(u,1) \neq 0$. Therefore, for each $v \in \mathcal{N}_\lambda^{\pm}$ and taking into account the Implicit Function Theorem \cite{drabek}, we obtain that there exist unique functionals $t^{n,\pm}: B_{\epsilon}(v) \to (1-\delta, 1 + \delta)$ such that $H(u, t^{n,\pm}(u) ) = 0$ for each $u \in B_{\epsilon}(v)$ where $\epsilon, \delta > 0$ are small and $B_{\epsilon}(v) = \{ u \in W^{1,\Phi}_0(\Omega) : \|u - v\| < \epsilon \}$.   This ends the proof.
	\end{proof}

	\begin{proposition}
		Suppose that assumptions $(\phi_1)-(\phi_4)$ and \ref{paper4f1}--\ref{paper4f3} are satisfied. Assume also that $\lambda > 0$. Then the sets $\mathcal{N}_{\lambda}^+$ and $\mathcal{N}_{\lambda}^-$ are $C^1$ manifolds.
	\end{proposition}
	\begin{proof}
		Consider the functional $R : W^{1,\Phi}_0(\Omega) \setminus \{0\} \to \mathbb{R}$ given by $R(u) = J_{\lambda}'(u)u$. It is not hard to prove that $R$ is in $C^1$ class. Recall also that $R'(u)u = J_{\lambda}''(u)(u,u) \neq 0$ holds for each $u \in \mathcal{N}_{\lambda}^+ \cup \mathcal{N}_{\lambda}^-$. Furthermore, $\mathcal{N}_{\lambda}^+ \cup \mathcal{N}_{\lambda}^- = R^{-1}(0)$. The desired result follows from the Implicit Function Theorem \cite{drabek}.
		This finishes the proof. 
	\end{proof}

	\begin{proposition}
		Suppose that assumptions $(\phi_1)-(\phi_4)$ and \ref{paper4f1}--\ref{paper4f4} are satisfied. Then the we obtain the following assertions:
		\begin{itemize}
			\item[i)] There holds $t_n(u) < t_e(u)$ for each $u \in W^{1,\Phi}_{0}(\Omega) \setminus \{0\}$.
			\item[ii)] There holds $\Lambda_e(u) < \Lambda_n(u)$ for each $u \in W^{1,\Phi}_{0}(\Omega) \setminus \{0\}$.
			\item[iii)] It holds that $0 < \lambda_* < \lambda^* < \infty$.
		\end{itemize}
	\end{proposition}
	\begin{proof}
		Firstly, we shall prove the item $i)$. It follows from a standard calculation that 
		\begin{equation*}
			R_n(tu) - R_e(tu) = \frac{t}{q} \dfrac{d}{dt} R_e(tu), t > 0, u \in W^{1,\Phi}_{0}(\Omega) \setminus \{0\}. 
		\end{equation*}
		In particular, $R_n(tu) > R_e(tu)$ for each $t \in (0, t_e(u))$. Furthermore,  $R_n(tu) = R_e(tu)$ if and only if $t = t_e(u)$. Analogously, we mention that  $R_n(tu) < R_e(tu)$ holds for each $t > t_e(u)$. In particular, $t_n(u) < t_e(u)$ for each $u \in W^{1,\Phi}_{0}(\Omega) \setminus \{0\}$. This ends the proof of item $i)$. 
		
		Now we shall prove the item $ii)$. According the previous item we obtain that 
		
		\begin{equation*}
			\Lambda_n(u) = \sup_{t > 0} R_n(tu) = R_n(t_n(u)u) > R_n(t_e(u)u) = R_e(t_e(u)u) = \Lambda_e(u), u \in W^{1,\Phi}_{0}(\Omega) \setminus \{0\}.
		\end{equation*}
		
		Now we shall prove the item $iii)$. At this stage, by using Proposition \ref{marcos}, there exists $u_n \in W^{1,\Phi}_{0}(\Omega) \setminus \{0\}$ such that $\Lambda_n(u_n) = \lambda^*$. As a consequence, for each $u \in W^{1,\Phi}_{0}(\Omega) \setminus \{0\}$, we infer that 
		\begin{equation*}
			\lambda^*=\Lambda_n(u_n) > \Lambda_e(u_n) \geq \Lambda_e(u).
		\end{equation*}
		In particular, we deduce that $\lambda^* > \lambda_*$.
		This ends the proof.	
	\end{proof}

	\begin{proposition}\label{n0}
		Suppose that assumptions $(\phi_1)-(\phi_4)$ and \ref{paper4f1}--\ref{paper4f4} are satisfied. Then we obtain the following assertions:
		\begin{itemize}
			\item[i)] For each $\lambda \in (0, \lambda^*)$ it holds $\mathcal{N}_{\lambda}^0 = \emptyset$.
			\item[ii)] It holds that $\mathcal{N}_{\lambda}^0 \neq  \emptyset$ for each $\lambda \geq \lambda^*$.
		\end{itemize}
	\end{proposition}
	\begin{proof}
		Initially, by using \eqref{b11} and \eqref{rel-J-R}, we obtain that $u\in \mathcal{N}_\lambda^0$ if, and only if, $R_n(u)=\lambda$ and $R_n'(u)u=0$. Now we shall prove the item $i)$.	
		Let us assume that $\lambda\in(0,\lambda^*)$. The proof follows arguing by contradiction. Suppose that there exists $u\in\mathcal{N}_{\lambda}^0$. This implies that $\lambda=R_n(u)$ and $t_n(u)=1$ is global maximum point of $t\mapsto R_n(tu)$. Therefore,
		$$0 < \lambda < \lambda^*=\inf_{v\neq 0}\Lambda_n(v)=\inf_{v\neq 0}\max_{t\geq 0}R_n(tv)\leq R_n(u)=\lambda.$$
		This is a contradiction proving the item $i)$.
		
		Now, we shall prove the item $ii)$. Firstly, we assume that $\lambda=\lambda^*$. In this case, the proof of item $ii)$ follows of Proposition \ref{marcos}-$v)$ and the first statement given above.
		
		It remains to consider the case $\lambda>\lambda^*$. We claim that there exists a sequence $(w_k)\subset W_0^{1,\Phi}(\Omega)$ such that $\Lambda_n(w_k)\to \infty$ as $k\to \infty$. In fact, taking into account that $W_0^{1,\Phi}(\Omega)$ is a Banach reflexive space, there exists $(u_k)\subset W_0^{1,\Phi}(\Omega)$ such that $\|u_k\|=1$ and $u_k\rightharpoonup 0$ where $u_k$ does not strong converges to the zero function in $W_0^{1,\Phi}(\Omega)$. Define the auxiliary sequence $w_k=t_n(u_k)u_k$. It follows from \eqref{est-ln} that
		$$\Lambda_n(w_k)=R_n(w_k)=\max_{t\geq 0}R_n(tw_k)\geq \ell\frac{p-m}{p-q}\frac{\min\left\{\|w_k\|^\ell,\|w_k\|^m\right\}}{\|w_k\|^q_{q,a}}=\ell\frac{p-m}{p-q}\frac{1}{\|w_k\|^q_{q,a}}\to \infty.$$
		Here was used the fact that $\|w_k\|^q_{q,a} \to 0$ as $k \to \infty$ due to the compact embedding $W^{1,\Phi}_0(\Omega)$ into $L^\ell(\Omega)$.
		Hence, $\Lambda_n(w_k) \to \infty$ as $k \to \infty$. In particular, for each $k \in \mathbb{N}$ large enough, we obtain that $\Lambda_n(w_k) > \lambda$ and $\Lambda_n(\overline{w}) = \lambda^* < \lambda$ for some $\overline{w} \in W^{1,\Phi}_0(\Omega)$, see Proposition \ref{marcos}. Furthermore, using the claim just above together with the Intermediate Value Theorem, there exists $w\in W_0^{1,\Phi}(\Omega)$ such that $\Lambda_n(w)=\lambda$. The last assertion implies that $w \in \mathcal{N}_\lambda^0$ proving that $\mathcal{N}_\lambda^0\neq \emptyset$. This ends the proof.	
	\end{proof}

	\begin{proposition}\label{coercive}
		Suppose that assumptions $(\phi_1)-(\phi_4)$ and \ref{paper4f1}--\ref{paper4f3} are satisfied. Assume also that $\lambda > 0$. Then the functional $J_{\lambda}$ is coercive in the Nehari manifold $\mathcal{N}_{\lambda}$, that is, $J_{\lambda}(u) \to \infty$ as $\|u\| \to \infty, u \in \mathcal{N}_{\lambda}$.
	\end{proposition}
	\begin{proof}
		Let $u \in \mathcal{N}_{\lambda}$ be fixed. It is easy to verify that 
		\begin{equation*}
			J_{\lambda}(u) = J_{\lambda}(u) - \dfrac{1}{p} J_{\lambda}'(u) u = \int_{\Omega} [\Phi(|\nabla u |) - \dfrac{1}{p} \phi(|\nabla u |) |\nabla u |^2] dx + \lambda\left(-\dfrac{1}{q} + \dfrac{1}{p} \right) \|u\|_{q,a}^q 
		\end{equation*}
		Recall also that $\ell\Phi(t) \leq \phi(t) t^2  \leq m \Phi(t), t \geq 0$. The last assertion can be proved using the hypothesis $(\phi_4)$ and the integration by parts. Furthermore, by using the last estimate together with the continuous embedding $W_0^{1,\Phi}(\Omega)\hookrightarrow L^{q}(\Omega)$, there exists $C > 0$ such that
		\begin{equation*}
			J_{\lambda}(u) \geq \left( 1 - \frac{m}{p} \right) \int_{\Omega} \Phi(|\nabla u |) dx - \lambda C \|u\|^q_{q,a} 
		\end{equation*}
		On the other hand, by using Proposition \ref{Narukawa}, we observe that there exists $c_1, c_2 > 0$ such that 
		\begin{equation*}
			c_1 \|u\|^\ell \leq \int_{\Omega} \Phi(|\nabla u |) dx \leq c_2 \|u\|^m, u \in W^{1,\Phi}_0(\Omega), \|u\| \geq 1.
		\end{equation*}
		As a consequence, 
		\begin{equation*}
			J_{\lambda}(u) \geq \left( 1 - \frac{m}{p} \right) c_1 \|u\|^\ell - \lambda C \|u\|^q_{q,a} \to \infty \,\, \mbox{as} \,\, \|u\| \to \infty, u \in \mathcal{N}_{\lambda}.
		\end{equation*}
		This finishes the proof. 
	\end{proof}

	\begin{proposition}
		Suppose that assumptions $(\phi_1)-(\phi_4)$ and \ref{paper4f1}--\ref{paper4f4} are satisfied. Assume also that $\lambda > 0$. Then $c_{\mathcal{N}_{\lambda}^+} < 0$. Furthermore, $\overline{\mathcal{N}_{\lambda}^+} = N^+ \cup \{0\}$ for each $\lambda \in (0, \lambda^*)$. 
	\end{proposition}
	\begin{proof}
		Firstly, we consider $u \in \mathcal{N}_{\lambda}^+$. Recall that $R_n(u) = \lambda$ and $t^{n,+}(u) = 1 < t_n(u) < t_e(u)$, see Proposition \ref{proje}. In particular, we see that $\lambda = R_n(u) > R_e(u)$. Hence, by using Remark \ref{rmk11}, we obtain that 
		\begin{equation*}
			c_{\mathcal{N}_{\lambda}^+} \leq J_\lambda(u) < J_{R_e(u)}(u) = 0.
		\end{equation*}
		Now, we observe that the proof of the fact that $\overline{\mathcal{N}_{\lambda}^+} = N^+ \cup \{0\}$ holds for each $\lambda \in (0, \lambda^*)$ follows arguing as was done in the proof of \cite[Lemma 2.3]{Carvalho3}. This finishes the proof. 
	\end{proof}

	\begin{proposition}\label{dist-zero}
		Suppose that assumptions $(\phi_1)-(\phi_4)$ and \ref{paper4f1}--\ref{paper4f4} are satisfied. Assume also that $\lambda > 0$. Then there exists $c > 0$ such that $\|u\| \geq c$ for each $u \in \mathcal{N}_{\lambda}^- \cup \mathcal{N}_{\lambda}^0$. Furthermore, $\mathcal{N}_{\lambda}^-$ and $\mathcal{N}_{\lambda}^0$  are closed sets in $W^{1,\Phi}_0(\Omega)$. 
	\end{proposition}
	\begin{proof}
		Let $u \in \mathcal{N}_{\lambda}^- \cup \mathcal{N}_{\lambda}^0$ be fixed. It is easy to see that 
		\begin{equation}\label{chave1}
			\int_{\Omega} \phi(|\nabla u |) |\nabla u |^2 dx = \lambda \|u\|_{q,a}^q + \|u\|_p^p . 
		\end{equation}
		Now, by using \eqref{ed2}, we observe that 
		\begin{equation}\label{chave2}
			0 \geq \int_{\Omega}
			\phi(|\nabla u|) |\nabla u|^2 + \phi'(|\nabla u|)|\nabla u|^3 dx  - \lambda(q -1) \|u\|_{q,a}^q  - (p-1) \|u\|_p^p.
		\end{equation} 
		Hence, taking into account \eqref{chave1} and \eqref{chave2}, we obtain
		\begin{equation*}
			\int_{\Omega} [(2 -q) \phi(|\nabla u |) |\nabla u |^2 + \phi'(|\nabla u|)|\nabla u|^3 ] dx \leq (p-q) \|u\|_p^p.
		\end{equation*}
		Now, using hypothesis $(\phi_4)$, we obtain that
		\begin{equation*}
			(\ell - q) \int_{\Omega} \phi(|\nabla u |) |\nabla u |^2  dx \leq (p-q) \|u\|_p^p.
		\end{equation*}
		Therefore, by using hypothesis $(\phi_4)$ and the embedding $W_0^{1,\Phi}(\Omega)\hookrightarrow L^{p}(\Omega)$, there exists $c_1 > 0$ such that
		\begin{equation*}
			\int_{\Omega} \Phi(|\nabla u |) dx \leq c_1 \|u\|^p.
		\end{equation*}
		Furthermore, by using  Proposition \ref{Narukawa}, we deduce that there exists a constant $c_2 > 0$ such that
		\begin{equation*}
			min(\|u\|^\ell, \|u\|^m)  \leq c_2 \|u\|^p.
		\end{equation*}
		Let us assume that $\|u\| \leq 1$. Under this condition we obtain that
		\begin{equation*}
			\|u\| \geq c_2^{p - \ell}, \,\, \mbox{for each} \,\, \|u\| \leq 1.
		\end{equation*}
		In general, we obtain also that 
		\begin{equation*}
			\|u\| \geq \min(1, c_2^{p - \ell}), u \in \mathcal{N}_{\lambda}^- \cup \mathcal{N}_{\lambda}^0.
		\end{equation*}
		The desired result follows by using a standard argument. We omit the details. 
	\end{proof}

	\begin{proposition}\label{strong-N+}
		Suppose that assumptions $(\phi_1)-(\phi_4)$ and \ref{paper4f1}--\ref{paper4f4} are satisfied. Assume also that $0<\lambda\leq \lambda^*$. Let 
		$(u_k) \in \mathcal{N}_{\lambda}^+$ be a minimizer sequence for the minimization problem given in \eqref{ee1}. Then there exists $u \in W^{1,\Phi}_0(\Omega) \setminus \{0\}$ such that $u_k \to u$ in $W^{1,\Phi}_0(\Omega)$. Furthermore, $c_{\mathcal{N}_{\lambda}^+} = J_{\lambda}(u) < 0$ and $u \in \mathcal{N}_{\lambda}^+ \cup \mathcal{N}_{\lambda}^0$.
	\end{proposition}
	\begin{proof}
		Firstly, by using Proposition \ref{coercive}, the sequence $(u_k)$ is bounded in $W^{1,\Phi}_0(\Omega)$. Up to a subsequence we assume that $u_k\rightharpoonup u$ in $W^{1,\Phi}_0(\Omega)$. Furthermore, we claim that $u\not\equiv 0$. In fact, given $v\in \mathcal{N}_\lambda^+$, we have that $\lambda=R_n(v)>R_e(v)$. The last estimate is equivalent to $J_{\lambda}(v)<0$. Since $u\mapsto J_\lambda(u)$ is weakly lower semicontinuous we obtain
		$$J_{\lambda}(u)\leq \liminf_{k\to\infty} J_{\lambda}(u_k)=\inf_{w \in \mathcal{N}_\lambda^+}J_{\lambda}(w)\leq J_{\lambda}(v)<0.$$
		The last estimates imply that the claim is satisfied. Now, we shall prove that $u_k\to u$ in $W_0^{1,\Phi}(\Omega)$. The proof follows arguing by contradiction. Let us assume that $u_k\not\to u$ in $W^{1,\Phi}_0(\Omega)$. Since $u\mapsto R_n(u)$ is weakly lower semicontinuous, we deduce that
		$$R_n(u)<\liminf_{k\to \infty} R_n(u_k)=\lambda\leq \lambda^*=\inf_{u\neq 0}\max_{t\geq 0}R_n(tu).$$
		Consequently, by using Proposition \ref{proje}, there exists $t^{n,+}(u)>0$ such that $t^{n,+}(u)u\in \mathcal{N}_\lambda^+$. Using again that $u\mapsto R_n(u)$ is weakly lower semicontinuous, we infer that $t^{n,+}(u)>1$. Moreover, $t\mapsto J_\lambda(tu)$ is decreasing in $[0,t^{n,+}(u)]$ and $v\mapsto J_\lambda(v)$ is weakly lower semicontinuous. Therefore,  we obtain $$J_\lambda(t^{n,+}(u)u)<J_\lambda(u)<\liminf_{k\to \infty}J_\lambda(u_k)=c_{\mathcal{N}_{\lambda}^+}.$$
		This is a contradiction proving that $u_k\to u$ in $W_0^{1,\Phi}(\Omega)$. Now, by using the fact that $u\neq 0$ and $J_\lambda''(u)(u,u)\geq 0$,  we infer also that $u \in \mathcal{N}_{\lambda}^+ \cup \mathcal{N}_{\lambda}^0$. This finishes the proof.	 
	\end{proof}
	
	\begin{proposition}\label{assim}
		Suppose that assumptions $(\phi_1)-(\phi_4)$ and \ref{paper4f1}--\ref{paper4f4} are satisfied.  Assume also that $0<\lambda\leq \lambda^*$. Let 
		$(v_k) \in \mathcal{N}_{\lambda}^-$ be a minimizer sequence for the minimization problem given in \eqref{ee2}. Then there exists $v \in W^{1,\Phi}_0(\Omega) \setminus \{0\}$ such that $v_k \to v$ in $W^{1,\Phi}_0(\Omega)$. Furthermore, $c_{\mathcal{N}_{\lambda}^-} = J_{\lambda}(v)$ and  $v \in \mathcal{N}_{\lambda}^- \cup \mathcal{N}_{\lambda}^0$.
	\end{proposition}
	\begin{proof}
		In virtue of Proposition \ref{coercive} the sequence $(v_k)$ is bounded. Without loss of generality we assume that $v_k\rightharpoonup v$ in $W^{1,\Phi}_0(\Omega)$. At this stage, we claim that $v\not\equiv 0$. The proof of this claim follows arguing by contradiction. Let us assume that $v\equiv 0$ is satisfied. Under these conditions we observe that $\|v_k\|_p,\|v_k\|_{q}\to 0$ as $k\to\infty$. Notice also that $(v_k) \in \mathcal{N}_\lambda^-$. It follows from Proposition \ref{Narukawa} and Proposition \ref{dist-zero} that
		$$\ell\min\{c^\ell,c^m\}\leq\ell\min\{\|v_k\|^\ell,\|v_k\|^m\}\leq\ell\int_\Omega\Phi(|\nabla v_k|)dx \leq \int_\Omega\phi(|\nabla v_k|)|\nabla v_k|^2dx=\|v_k\|_p^p+\lambda\|v_k\|_{q}^q\to 0.$$
		This is a contradiction proving the desired claim. 
		
		Now, we shall prove that $v_k\to v$ in $W_0^{1,\Phi}(\Omega)$. Arguing as was done in the proof of Proposition \ref{strong-N+}, there exists $t^{n,-}(v)>0$ such that $t^{n,+}(v)v\in \mathcal{N}_\lambda^-$. Since $u\to R_n(u)$ is weakly lower semicontinuous, we infer that $t^{n,-}(v)<1$. Furthermore, by using the fact that $v\mapsto R_n(tv)$ is increasing in $[t^{n,-}(v),\infty)$ and $w\mapsto J_\lambda(w)$ is weakly lower semicontinuous, we obtain
		$$J_\lambda(t^{n,-}(v)v)<J_\lambda(v)<\liminf_{k\to \infty}J_\lambda(v_k)=c_{\mathcal{N}_{\lambda}^-}.$$
		The last estimates do not make sense proving the strong convergence is satisfied. Now, by using the fact that $v\neq 0$ and $J_\lambda''(v)(v,v)\leq 0$,  we conclude that $u \in \mathcal{N}_{\lambda}^- \cup \mathcal{N}_{\lambda}^0$.  This finishes the proof.
	\end{proof}

	\begin{proposition}\label{lagrange}
		Suppose that assumptions $(\phi_1)-(\phi_4)$ and \ref{paper4f1}--\ref{paper4f4} are satisfied.  Assume also that $\lambda \in (0, \lambda^*)$. Let $u \in \mathcal{N}_{\lambda}^- \cup \mathcal{N}_{\lambda}^0$ and $v \in \mathcal{N}_{\lambda}^- \cup \mathcal{N}_{\lambda}^0$ the minimizers for the Problems \eqref{ee1} and \eqref{ee2}, respectively. Then $u$ and $v$ are critical points for the functional $J_{\lambda}$, that is, $J_{\lambda}'(u) \psi = 0$ and $J_{\lambda}'(v) \psi = 0$ hold for each $\psi \in W^{1,\Phi}_0(\Omega)$.
	\end{proposition}
	\begin{proof}
		Recall that $\mathcal{N}_{\lambda}^0 = \emptyset$ for each $\lambda \in (0, \lambda^*)$, see Proposition \ref{n0}. The last assertion implies that $u \in \mathcal{N}_{\lambda}^+$ and $v \in \mathcal{N}_{\lambda}^-$. On the other hand, by using Lagrange Multipliers Theorem \cite{drabek}, we obtain that there exists $\theta_1, \theta_2 \in \mathbb{R}$ such that
		\begin{equation}\label{ajei}
			J_{\lambda}'(u) \psi = \theta_1 R'(u) \psi, J_{\lambda}'(v) \psi = \theta_2 R'(v) \psi, \psi \in W^{1,\Phi}_0(\Omega), 
		\end{equation}
		where $R: W^{1,\Phi}_0(\Omega) \setminus \{0\} \to \mathbb{R}$ is given by $R(u) = J_{\lambda}'(u)u$.  Notice also that $R'(u)u = J_{\lambda}''(u)(u,u) > 0$ and $R'(v)v = J_{\lambda}''(v)(v,v) < 0$. The last estimates together with \eqref{ajei} say that $\theta_1 = \theta_2 = 0$. This finishes the proof. 
	\end{proof}
	
	\begin{proposition}\label{15}
		Suppose that assumptions $(\phi_1)-(\phi_4)$ and \ref{paper4f1}--\ref{paper4f4} are satisfied.  Assume also that $\lambda > 0$. Let $v\in W^{1,\Phi}_0(\Omega)$ be the function obtained in the Proposition \ref{assim}. Then, we obtain the following statements:
		\begin{itemize}
			\item [i)] Assume that $\lambda \in (0,\lambda_*)$. Then $c_{\mathcal{N}_{\lambda}^+}=J_{\lambda}(v)>0$.
			\item[ii)] Suppose $\lambda=\lambda_*$. Then  $c_{\mathcal{N}_{\lambda}^+}=J_{\lambda}(v)=0$.
			\item[iii)] Suppose $\lambda \in(\lambda_*, \infty)$. Then $c_{\mathcal{N}_{\lambda}^+}=J_{\lambda}(v)<0$.
		\end{itemize}
	\end{proposition}
	\begin{proof}
		Here the main idea is to use the fact that the function $\lambda \mapsto J_{\lambda}(u)$ is decreasing for each $u \in W^{1,\Phi}_0(\Omega) \setminus \{0\}$. Hence, by using the same ideas discussed in the proof of \cite[Theorem 1.2]{yavdat0}, the proof follows. We omit the details.
	\end{proof}
	
	\section{The case $\lambda = \lambda^*$}
	In this section we shall consider the Problem \eqref{pi} assuming that $\lambda = \lambda^*$. As was told before we observe that $\mathcal{N}_\lambda^0$ is not empty, see Proposition \ref{n0}. Hence, we need to control the behavior of $J_\lambda$ with $\lambda =\lambda^*$. Under these conditions we are able to use a sequence $(\lambda_k)_{k \in \mathbb{N}} \in \mathbb{R}$ such that $\lambda_k < \lambda^*$ for each $k \in \mathbb{N}$ and $\lambda_k \to \lambda^*$ as $k \to \infty$. This is the main ingredient in order to prove our existence result for $\lambda = \lambda^*$. Firstly, we consider the following useful result:
	\begin{proposition}\label{sol1}
		Suppose that assumptions $(\phi_1)-(\phi_4)$, \ref{paper4f1}--\ref{paper4f4} are satisfied. Assume also that $\lambda = \lambda^*$. Then Problem \eqref{pi} has a nontrivial solution $u \in \mathcal{N}^+_{\lambda^*} \cup \mathcal{N}_{\lambda^*}^0$.
	\end{proposition}
	\begin{proof}
		Let us consider a sequence $(\lambda_k)_{k \in \mathbb{N}} \in \mathbb{R}$ such that $\lambda_k < \lambda^*$ for each $k \in \mathbb{N}$ and $\lambda_k \to \lambda^*$ as $k \to \infty$. According to Theorem \ref{th1} there exists a sequence $(u_k) \in \mathcal{N}^+_{\lambda_k}$ in such way that 
		\begin{equation}\label{sacouu}
			\int_{\Omega} \phi(|\nabla u_k|) \nabla u_k \nabla \psi dx =  \int_{\Omega} [\lambda_k a(x) |u_k|^{q-2}u_k + |u_k|^{p-2}u_k] \psi dx, \psi \in W^{1, \Phi}_0(\Omega).
		\end{equation}
		Furthermore, we observe that
		$c_{\lambda_k} = J_{\lambda_k}(u_k), k \in \mathbb{N}$. It is not hard to see that $\lambda \mapsto C_{\mathcal{N}_\lambda^+}$ is a decreasing function. Hence, we obtain that 
		\begin{eqnarray}\label{a0}
			C_{\mathcal{N}_{\lambda_1}^+} &\geq& C_{\mathcal{N}_{\lambda_k}^+} = J_\lambda(u_k) \nonumber \\
			&=& J_{\lambda}(u_k) - \dfrac{1}{p} J_{\lambda}'(u_k) u_k = \int_{\Omega} [\Phi(|\nabla u_k |) - \dfrac{1}{p} \phi(|\nabla u_k |) |\nabla u_k |^2] dx + \lambda\left(-\dfrac{1}{q} + \dfrac{1}{p} \right) \|u_k\|_q^q.\nonumber
		\end{eqnarray}
		In particular, $(u_k)$ is now bounded in $W^{1,\Phi}_0(\Omega)$. Hence there exists $u \in W^{1,\Phi}_0(\Omega)$ such that $u_k \rightharpoonup u$ in $W^{1,\Phi}_0(\Omega)$. At this stage, by using the same ideas employed in the proof of Proposition \ref{strong-N+}, we obtain that $u_k \to u$ in $W^{1,\Phi}_0(\Omega)$. Therefore, taking into account that $J$ is in $C^1$ class, we deduce that $u \in \mathcal{N}^+_{\lambda^*} \cup \mathcal{N}^0_{\lambda^*}$. Here was used also the fact that $u \mapsto J''(u)(u,u)$ is in $C^0$ class. In the same way, by using \eqref{sacouu} and the strong converge given just above, we obtain that
		\begin{equation*}
			\int_{\Omega} \phi(|\nabla u|) \nabla u \nabla \psi dx =  \int_{\Omega} [\lambda^* a(x) |u|^{q-2}u + |u|^{p-2}u] \psi dx, \psi \in W^{1, \Phi}_0(\Omega).
		\end{equation*} 
		This ends the proof. 	
	\end{proof}
	
	Similarly, we also consider the following result:
	
	\begin{proposition}\label{sol2}
		Suppose that assumptions $(\phi_1)-(\phi_4)$, \ref{paper4f1}--\ref{paper4f4} are satisfied. Assume also that $\lambda = \lambda^*$. Then Problem \eqref{pi} admits at least one nontrivial solution $v \in \mathcal{N}_{\lambda^*}^- \cup \mathcal{N}_{\lambda^*}^0$.
	\end{proposition}
	\begin{proof}
		Let us consider a sequence $(\lambda_k)_{k \in \mathbb{N}} \in \mathbb{R}$ such that $\lambda_k < \lambda^*$ for each $k \in \mathbb{N}$ and $\lambda_k \to \lambda^*$ as $k \to \infty$. According to Propositions \ref{assim} and \ref{lagrange} there exists a sequence $(v_k) \in \mathcal{N}^-_{\lambda_k}$ in such way that 
		\begin{equation*}\label{sacou}
			\int_{\Omega} \phi(|\nabla u_k|) \nabla u_k \nabla \psi dx =  \int_{\Omega} [\lambda_k a(x) |v_k|^{q-2}v_k + |v_k|^{p-2}v_k] \psi dx, \psi \in W^{1, \Phi}_0(\Omega).
		\end{equation*}	
		Hence, the rest of the proof follows the same lines discussed in the proof of Proposition \ref{sol2}. We omit the details. 	
	\end{proof}

	\section{The case $\lambda > \lambda^*$}
	
	Now we shall consider the case $\lambda > \lambda^*$. The main feature here is to consider $\overline{\lambda} > \lambda^*$ which are defined in \eqref{lambdabarra}. In particular, for each $\lambda^* < \lambda < \overline{\lambda}$, we are able to prove the following result: 
	\begin{corollary}\label{esquema}
		Suppose that assumptions $(\phi_1)-(\phi_4)$ and \ref{paper4f1}--\ref{paper4f4} are satisfied. Then, for each $\lambda^* < \lambda < \overline{\lambda}$, the Problem \eqref{pi} does not admit any nontrivial solution in $\mathcal{N}_{\lambda}^0$.
	\end{corollary}
	\begin{proof}
		It is important to observe that $\overline{\lambda} < \infty$. The proof follows arguing by contradiction. Let us assume that $u \in \mathcal{N}_{\lambda}^0$ is a weak nontrivial solution for the Problem \eqref{pi}. Hence, we obtain a contradiction due to the fact that $\lambda < \overline{\lambda}$. This ends the proof.
	\end{proof}
	
	From now on, for each $\lambda^* < \lambda < \overline{\lambda}$, we can find two nontrivial solutions for the Problem \eqref{pi}. Namely, we can prove the following result: 
	\begin{proposition}\label{sol11}
		Suppose that assumptions $(\phi_1)-(\phi_4)$, \ref{paper4f1}--\ref{paper4f4} are satisfied. Assume also that $\lambda^* < \lambda < \overline{\lambda}$. Then Problem \eqref{pi} has a nontrivial solution $u \in \mathcal{N}^+_{\lambda}$.	
	\end{proposition}
	\begin{proof}
		The proof follows the same ideas discussed in the proof of Proposition \ref{sol1}. Recall also that Problem \eqref{pi} does not admit any nontrivial solution in $\mathcal{N}_\lambda^0$, see Corollary \ref{esquema}. We omit the details. 
	\end{proof}
	\begin{proposition}\label{sol22}
		Suppose that assumptions $(\phi_1)-(\phi_4)$, \ref{paper4f1}--\ref{paper4f4} are satisfied. Assume also that $\lambda = \lambda^*$. Then Problem \eqref{pi} admits at least one nontrivial solution $v \in \mathcal{N}_{\lambda^*}^- \cup \mathcal{N}_{\lambda^*}^0$.
	\end{proposition}
	\begin{proof}
		The proof follows the lines discussed in the proof of Proposition \ref{sol2}. It is important to emphasize that Problem \eqref{pi} does not admit any nontrivial solution in $\mathcal{N}_\lambda^0$, see Corollary \ref{esquema}. We omit the details. 
	\end{proof}

	\section{The proof of our main theorems}

	\subsection{The proof of Theorem \ref{th1}} 
	
	According to Propositions \ref{marcos} and \ref{marcos2} we know that $0 < \lambda_* < \lambda^* < \infty$. Let $\lambda \in (0, \lambda^*)$ be fixed. Now, by using Propositions \ref{coercive} and \ref{strong-N+}, we obtain that exists $u \in \mathcal{N}_{\lambda}^+$ such that $J_{\lambda}(u) = c_{\mathcal{N}_{\lambda}}^+ < 0$. Furthermore, taking into account Proposition \ref{n0} and Proposition \ref{lagrange}, we obtain that $u \in \mathcal{N}^+_\lambda$ is a critical point for the energy functional $J_{\lambda}$. Now, by using a standard argument, we obtain also that $u$ is a ground state solution for the Problem \eqref{pi}. Recall that $J_\lambda(w) = J_\lambda(|w|)$ for each $w \in W^{1,\Phi}_0(\Omega)$.  As a consequence, we can assume that $u \geq 0$ in $\Omega$. Now, by using the Strong Maximum Principle, we infer that $u > 0$ in $\Omega$, see \cite[Theorem 1.7]{ed}. In view of \ref{ed1} and \ref{ed2} we also mention that 
	\begin{equation*}
		\int_{\Omega} 
		(p - 2)\phi(|\nabla u|) |\nabla u|^2 - \phi'(|\nabla u|)|\nabla u|^3 dx \leq \lambda( p - q) \int_{\Omega} a(x) |u|^q dx, u \in \mathcal{N}^+_\lambda.
	\end{equation*}
	The last assertion together with the H\"older inequality and the embedding $W_0^{1,\Phi}(\Omega)\hookrightarrow L^{\ell}(\Omega)$ imply that 
	\begin{equation*}
		\int_{\Omega} \Phi(|\nabla u|) dx \leq c \lambda \|u\|^q.
	\end{equation*}
	Hence, by using Proposition \ref{Narukawa}, we obtain that 
	$
		\|u\|^{\ell - q} \leq c \lambda \,\, \mbox{or} \,\, \|u\|^{m - q} \leq c \lambda.
	$
	Hence, for each weak solution $u \in \mathcal{N}_\lambda^+$, we obtain that $\|u\| \to 0$ as $\lambda \to 0$. This ends the proof. 
	
	\subsection{The proof of Theorem \ref{th2}} Firstly, by using Propositions \ref{coercive}, \ref{assim} and \ref{lagrange}, we infer that there exists $v \in \mathcal{N}_{\lambda}^-$ such that $J_{\lambda}(v) = c_{\mathcal{N}_{\lambda}^-}$ and $J_{\lambda}'(v) \psi = 0$ holds for each $\phi \in W^{1,\Phi}_{0}(\Omega)$. Without any loss of generality, we assume that $v \geq 0$ in $\Omega$. Once again we can apply the Strong Maximum Principle proving that $v > 0$  in $\Omega$, see \cite[Theorem 1.7]{ed}. The desired result follows taking into account Proposition \ref{15}. We omit the details.

	\subsection{The proof of Corollary \ref{cor}} 
	The proof of Corollary \ref{cor} follows immediately from Theorems \ref{th1} and \ref{th2}.
	
	\subsection{The proof of Theorem \ref{th3}} Initially, we shall mention that $\lambda = \lambda^*$. Now, by using Proposition \ref{sol1}, we infer that 
	there exists a weak solution $u \in \mathcal{N}_{\lambda}^+ \cup \mathcal{N}_{\lambda}^0$ for the Problem \eqref{pi}. Similarly, by using Proposition \ref{sol2}, we obtain a weak solution  $v \in \mathcal{N}_{\lambda}^- \cup \mathcal{N}_{\lambda}^0$ for the Problem \eqref{pi}. Recall that Problem \eqref{pi} does not admit weak solution in the set $\mathcal{N}_\lambda^0$. As a consequence, we obtain that Problem \eqref{pi} has at least two nontrivial solutions whenever $\lambda = \lambda^*$. This ends the proof. 
	
	\subsection{The proof of Corollary \ref{cor2}} Initially, we assume that $ \lambda^* < \overline{\lambda}$. The main idea here is to apply the same ideas devoted to the proof of Theorem \ref{th3}. Recall that $\lambda \in [\lambda^*, \overline{\lambda})$ is verified. In particular, we observe that $\mathcal{N}_{\lambda}^0 \neq \emptyset$. On the other hand, for each $\lambda \in  [\lambda^*, \overline{\lambda})$, we know that $\mathcal{N}_{\lambda}^0$ does not admit any nontrivial solution for the Problem \ref{pi} whenever $\lambda \in [\lambda^*, \overline{\lambda})$. As a consequence, we can argue as was done in the proof of Theorem \ref{th3} obtaining two positive solutions $u \in \mathcal{N}_{\lambda}^+$ and $v \in \mathcal{N}_{\lambda}^-$, see for instance Propositions \ref{sol11} and \ref{sol22}. This ends the proof. 
	
	\subsection{The proof of Corollary \ref{cor3}}
	The proof of Corollary \ref{cor3} follows using the same arguments employed in the proof of Corollary  \ref{cor2}. The main feature here is to consider a sequence $(\lambda_k)_{k \in \mathbb{N}} \in \mathbb{R}$ such that $\lambda_k < \overline{\lambda}$ for each $k \in \mathbb{N}$ and $\lambda_k \to \overline{\lambda}$ as $k \to \infty$. It is important o stress that Problem \eqref{pi} has at least one solution $u_k \in \mathcal{N}_{\lambda}^+$ for each $\lambda = \lambda_k \in [\lambda^*, \overline{\lambda})$, see Corollary \ref{cor2}. Under these conditions, we can argue as was done in the proof of Corollary \ref{cor2}. Here we also refer the reader to Propositions \ref{sol11} and \ref{sol22}. As a consequence, Problem \eqref{pi} has at least one positive solution for $\lambda = \overline{\lambda}$. This finishes the proof.

\end{document}